\documentclass[11pt,a4paper]{amsart}
\pdfoutput=1
\usepackage{graphicx}
\usepackage{dsfont}
\usepackage{hyperref}
\usepackage{verbatim}

\newtheorem{theorem}{Theorem}[section]
\newtheorem{corollary}[theorem]{Corollary}
\newtheorem{proposition}[theorem]{Proposition}
\newtheorem{lemma}[theorem]{Lemma}

\theoremstyle{definition}
\newtheorem{definition}[theorem]{Definition}
\newtheorem{remark}[theorem]{Remark}

\numberwithin{equation}{section}

\newcommand{\IP}{\mathop{\null\mathds{P}}\nolimits}

\newcommand{\I}{\mathop{\null\mathds{1}}\nolimits}

\newcommand{\IN}{\mathds{N}}

\newcommand{\IZ}{\mathds{Z}}
\newcommand{\IC}{\mathds{C}}

\newcommand{\dist}{d}
\newcommand{\Ext}{\text{\textnormal{Ext}}}
\newcommand{\Hull}{\text{\textnormal{Hull}}}
\newcommand{\obound}{\partial_{\text{\textnormal{o}}}}

\newcommand{\Loop}{\text{\textnormal{loop}}}
\newcommand{\diam}{\text{\textnormal{diam}}}

\newcommand{\distinf}{d_{\infty}}

\newcommand{\BLS}{\mathcal{L}}

\begin{document}
\title[Random walk loop soups and conformal loop ensembles]{Random walk loop soups and\\ conformal loop ensembles}
\date{\today}
\author{Tim van de Brug}
\author{Federico Camia}
\author{Marcin Lis}

\address{Tim van de Brug\\
Department of Mathematics\\
VU University Amsterdam\\
De Boelelaan 1081a\\
1081 HV Amsterdam\\
The Netherlands}
\email{t.vande.brug\,@\,vu.nl}
\address{Federico Camia\\
Department of Mathematics\\
VU University Amsterdam\\
De Boelelaan 1081a\\
1081 HV Amsterdam\\
The Netherlands\\
and
New York University Abu Dhabi\\
PO Box 129188\\
Saadiyat Island\\
Abu Dhabi\\
United Arab Emirates}
\email{f.camia\,@\,vu.nl}
\address{Marcin Lis\\
Mathematical Sciences\\
Chalmers University of Technology and University of Gothenburg\\
SE-41296 Gothenburg\\
Sweden}
\email{marcinl\,@\,chalmers.se}

\begin{abstract}
The random walk loop soup is a Poissonian ensemble of lattice loops; it has been extensively studied because of its connections to the discrete Gaussian free field, but was originally introduced by Lawler and Trujillo Ferreras as a discrete version of the Brownian loop soup of Lawler and Werner, a conformally invariant Poissonian ensemble of planar loops with deep connections to conformal loop ensembles (CLEs) and the Schramm-Loewner evolution (SLE).

Lawler and Trujillo Ferreras showed that, roughly speaking, in the continuum scaling limit, ``large'' lattice loops from the random walk loop soup converge to ``large'' loops from the Brownian loop soup. Their results, however, do not extend to clusters of loops, which are interesting because the connection between Brownian loop soup and CLE goes via cluster boundaries. In this paper, we study the scaling limit of clusters of ``large'' lattice loops, showing that they converge to Brownian loop soup clusters. In particular, our results imply that the collection of outer boundaries of outermost clusters composed of ``large'' lattice loops converges to CLE.
\end{abstract}


\maketitle

\section{Introduction}

Several interesting models of statistical mechanics, such as percolation and the Ising
and Potts models, can be described in terms of clusters. In two dimensions and at the
critical point, the scaling limit geometry of the boundaries of such clusters is known
(see \cite{CamNew06,CamNew07,CamNew08,CheDum,Smi01}) or conjectured (see
\cite{KagNie,Smi10}) to be described by some member of the one-parameter family of Schramm-Loewner
evolutions (SLE$_{\kappa}$ with $\kappa>0$) and related conformal loop ensembles
(CLE$_{\kappa}$ with $8/3<\kappa<8$). What makes SLEs and CLEs natural candidates
is their conformal invariance, a property expected of the scaling limit of two-dimensional
statistical mechanical models at the critical point.
SLEs can be used to describe the scaling limit of single interfaces; CLEs are collections
of loops and are therefore suitable to describe the scaling limit of the collection of all
macroscopic boundaries at once. For example, the scaling limit of the critical percolation
exploration path is SLE$_6$ \cite{CamNew07,Smi01}, and the scaling limit of the
collection of all critical percolation interfaces in a bounded domain is CLE$_6$ \cite{CamNew06,CamNew08}.

For $8/3 < \kappa \leq 4$, CLE$_\kappa$ can be obtained \cite{SheWer} from the Brownian
loop soup, introduced by Lawler and Werner \cite{LawWer} (see Section \ref{sec:definitions} for a definition), as we explain below. A sample of the Brownian loop soup in a bounded domain $D$ with intensity $\lambda>0$ is the collection of loops contained in $D$ from a Poisson realization of a conformally invariant intensity measure $\lambda\mu$. When $\lambda \leq 1/2$, the loop soup is composed of disjoint clusters of loops \cite{SheWer} (where a cluster is a maximal collection of loops that intersect each other). When $\lambda>1/2$, there is a unique cluster~\cite{SheWer} and the set of points not surrounded by a loop is totally disconnected (see~\cite{BroCam}). Furthermore, when $\lambda \leq 1/2$, the outer boundaries of the outermost loop soup clusters are distributed like conformal loop ensembles (CLE$_{\kappa}$) \cite{She,SheWer,Wer03} with $8/3 < \kappa \leq 4$. More precisely, if $8/3 < \kappa \leq 4$, then $0 < (3 \kappa -8)(6 - \kappa) / 4 \kappa \leq 1/2$ and the collection of all outer boundaries of the outermost clusters of the Brownian loop soup with intensity $\lambda = (3 \kappa -8)(6 - \kappa)/4 \kappa$ is distributed like $\text{CLE}_{\kappa}$ \cite{SheWer}. For example, the continuum scaling limit of the collection of all macroscopic outer boundaries of critical Ising spin clusters is conjectured to correspond to $\text{CLE}_3$ and to a Brownian loop soup with $\lambda=1/4$.

We note that most of the existing literature, including \cite{SheWer}, contains an error in the correspondence between $\kappa$ and the loop soup intensity $\lambda$. The error can be traced back to the choice of normalization of the (infinite) Brownian loop measure $\mu$. (We thank Gregory Lawler for discussions on this topic.) With the normalization used in this paper, which coincides with the one in the original definition of the Brownian loop soup \cite{LawWer}, for a given $8/3 < \kappa \leq 4$, the corresponding value of the loop soup intensity $\lambda$ is half of that given in \cite{SheWer} -- see, for example, Section 6 of \cite{CamGanKle} for a discussion of this and of the relation between $\lambda$ and the central charge of the Brownian loop soup.

In \cite{LawTru} Lawler and Trujillo Ferreras introduced the random walk loop soup as a discrete version of the Brownian loop soup, and showed that, under Brownian scaling, it converges in an appropriate sense to the Brownian loop soup. The authors of \cite{LawTru} focused on individual loops, showing that, with probability going to 1 in the scaling limit, there is a one-to-one correspondence between ``large'' lattice loops from the random walk loop soup and ``large''  loops from the Brownian loop soup such that corresponding loops are close.

In \cite{Lej10} Le Jan showed that the random walk loop soup has remarkable connections with the discrete Gaussian free field, analogous to Dynkin's isomorphism \cite{Dyn84a,Dyn84b} (see also \cite{BryFroSpe}). Such considerations have prompted an extensive analysis of more general versions of the random walk loop soup (see e.g.\ \cite{Lej11,Szn}).

As explained above, the connection between the Brownian loop soup and SLE/CLE goes through its loop clusters and their boundaries. In view of this observation, it is interesting to investigate whether the random walk loop soup converges to the Brownian loop soup in terms of loop clusters and their boundaries, not just in terms of individual loops, as established by Lawler and Trujillo Ferreras \cite{LawTru}. This is a natural and nontrivial question, due to the complex geometry of the loops involved and of their mutual overlaps.

\begin{figure}
	\begin{center}
		\includegraphics[scale=1]{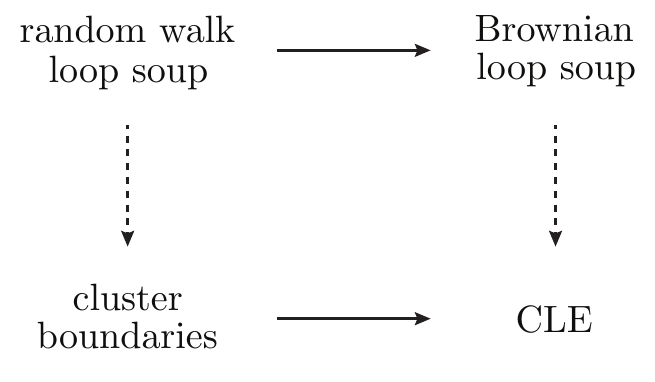}
	\end{center}
	\caption{Schematic diagram of relations between discrete and continuous loop soups and their cluster boundaries. Horizontal arrows indicate a scaling limit. In this paper we show the convergence corresponding to the bottom horizontal arrow.}
	\label{fig:diagram}
\end{figure}

In this paper, we consider random walk loop soups from which the ``vanishingly small'' loops have been removed and establish convergence of their clusters and boundaries, in the scaling limit, to the clusters and boundaries of the corresponding Brownian loop soups (see Figure \ref{fig:diagram}).
We work in the same set-up as \cite{LawTru}, which in particular means that the number of loops of the random walk loop soup after cut-off diverges in the scaling limit. We use tools ranging from classical Brownian motion techniques to recent loop soup results. Indeed, properties of planar Brownian motion as well as properties of CLEs play an important role in the proofs of our results.

We note that, while this paper was under review, a substantial improvement of our main result on the scaling limit of the random walk loop soup was announced by Lupu \cite{Lup}. The result announced appears to use our convergence result in a crucial way, combined with a coupling between the random walk loop soup and the Gaussian free field, and would give the convergence of the random walk loop soup to the Brownian loop soup keeping all loops.

\section{Definitions and main result}\label{sec:definitions}

We recall the definitions of the Brownian loop soup and the random walk loop soup. A \emph{curve} $\gamma$ is a continuous function $\gamma: [0,t_{\gamma}] \to \IC$, where $t_{\gamma}<\infty$ is the time length of $\gamma$. A \emph{loop} is a curve with $\gamma(0)=\gamma(t_{\gamma})$. A planar Brownian loop of time length $t_0$ started at $z$ is the process $z + B_t - (t/t_0) B_{t_0}$, $0\leq t\leq t_0$, where $B$ is a planar Brownian motion started at $0$.
The Brownian bridge measure $\mu^{\sharp}_{z,t_0}$ is a probability measure on loops, induced by a planar Brownian loop of time length $t_0$ started at $z$. The (rooted) Brownian loop measure $\mu$ is a measure on loops, given by
\[
\mu(C) = \int_{\IC} \int_0^{\infty} \frac{1}{2\pi t_0^2} \mu^{\sharp}_{z,t_0}(C) dt_0 dA(z),
\]
where $C$ is a collection of loops and $A$ denotes two-dimensional Lebesgue measure, see Remark 5.28 of \cite{Law}. For a domain $D$ let $\mu_D$ be $\mu$ restricted to loops which stay in $D$.

The (rooted) \emph{Brownian loop soup} with intensity $\lambda\in(0,\infty)$ in $D$ is a Poissonian realization from the measure $\lambda\mu_D$. The Brownian loop soup introduced by Lawler and Werner \cite{LawWer} is obtained by forgetting the starting points (roots) of the loops. The geometric properties we study in this paper are the same for both the rooted and the unrooted version of the Brownian loop soup. Let $\mathcal{L}$ be a Brownian loop soup with intensity $\lambda$ in a domain $D$, and let $\mathcal{L}^{t_0}$ be the collection of loops in $\mathcal{L}$ with time length at least $t_0$.

The (rooted) random walk loop measure $\tilde \mu$ is a measure on nearest neighbor loops in $\IZ^2$, which we identify with loops in the complex plane by linear interpolation. For a loop $\tilde\gamma$ in $\IZ^2$, we define
\[
\tilde\mu(\tilde\gamma) = \frac{1}{t_{\tilde\gamma}} 4^{-t_{\tilde\gamma}},
\]
where $t_{\tilde\gamma}$ is the time length of $\tilde\gamma$, i.e.\ its number of steps. The (rooted) \emph{random walk loop soup} with intensity $\lambda$ is a Poissonian realization from the measure $\lambda\tilde\mu$. For a domain $D$ and positive integer $N$, let $\mathcal{\tilde L}_N$ be the collection of loops $\tilde\gamma_N$ defined by $\tilde\gamma_N(t) = N^{-1} \tilde \gamma(2N^2 t)$, $0\leq t\leq t_{\tilde\gamma} / (2N^2)$, where $\tilde\gamma$ are the loops in a random walk loop soup with intensity $\lambda$ which stay in $ND$. Note that the time length of $\tilde\gamma_N$ is $t_{\tilde\gamma}/(2N^2)$. Let $\mathcal{\tilde L}_N^{t_0}$ be the collection of loops in $\mathcal{\tilde L}_N$ with time length at least $t_0$.

We will often identify curves and processes with their range in the complex plane, and a collection of curves $C$ with the set in the plane $\bigcup_{\gamma\in C} \gamma$. For a bounded set $A$, we write $\Ext A$ for the \emph{exterior} of~$A$, i.e.\ the unique unbounded connected component of
$\IC \setminus \overline{A}$. By $\Hull A$, we denote the \emph{hull} of $A$, which is the complement of $\Ext A$. We write $\obound A$ for the topological boundary of $\Ext A$, called the \emph{outer boundary} of $A$. Note that $\partial A \supset \obound A = \partial \Ext A = \partial \Hull A$. For sets $A,A'$, the \emph{Hausdorff distance} between $A$ and $A'$ is given by
\[
d_H(A,A') = \inf \{ \delta>0 : A\subset (A')^{\delta} \text{ and } A' \subset A^{\delta} \},
\]
where $A^{\delta} = \bigcup_{x\in A} B(x;\delta)$ with $B(x;\delta) = \{ y: |x-y| < \delta \}$.

Let $\mathcal{A}$ be a collection of loops in a domain $D$. A \emph{chain of loops} is a sequence of loops, where each loop intersects the loop which follows it in the sequence. We call $C \subset \mathcal{A}$ a \emph{subcluster} of $\mathcal{A}$ if each pair of loops in $C$ is connected via a finite chain of loops from $C$. We say that $C$ is a finite subcluster if it contains a finite number of loops. A subcluster which is maximal in terms of inclusion is called a \emph{cluster}. A cluster $C$ of $\mathcal{A}$ is called \emph{outermost} if there exists no cluster $C'$ of $\mathcal{A}$ such that $C'\neq C$ and $\text{Hull} C \subset \text{Hull} C'$. The \emph{carpet} of $\mathcal{A}$ is the set $D\setminus \bigcup_C (\Hull C \setminus \partial_o C)$, where the union is over all outermost clusters $C$ of $\mathcal{A}$. For collections of subsets of the plane $\mathcal{A},\mathcal{A}'$, the \emph{induced Hausdorff distance} is given by
\begin{align*}
d_H^*(\mathcal{A},\mathcal{A}') = &\inf \{ \delta>0 : \forall A\in\mathcal{A} \,\exists A'\in\mathcal{A}' \text{ such that } d_H(A,A')<\delta, \\
& \text{and } \forall A'\in\mathcal{A}' \,\exists A\in\mathcal{A} \text{ such that } d_H(A,A')<\delta \}.
\end{align*}

The main result of this paper is the following theorem:

\begin{theorem}\label{mainthmBLS}
Let $D$ be a bounded, simply connected domain, take $\lambda\in(0,1/2]$ and $16/9<\theta<2$. As $N\to\infty$,
\begin{itemize}
\item[(i)] the collection of hulls of all outermost clusters of $\mathcal{\tilde L}_N^{N^{\theta-2}}$ converges in distribution to the collection of hulls of all outermost clusters of~$\mathcal{L}$, with respect to $d_H^*$,
\item[(ii)] the collection of outer boundaries of all outermost clusters of $\mathcal{\tilde L}_N^{N^{\theta-2}}$ converges in distribution to the collection of outer boundaries of all outermost clusters of~$\mathcal{L}$, with respect to $d_H^*$,
\item[(iii)] the carpet of $\mathcal{\tilde L}_N^{N^{\theta-2}}$ converges in distribution to the carpet of $\mathcal{L}$, with respect to $d_H$.
\end{itemize}
\end{theorem}

As an immediate consequence of Theorem \ref{mainthmBLS} and the loop soup construction of conformal loop ensembles by Sheffield and Werner \cite{SheWer}, we have the following corollary:

\begin{corollary}\label{corcle}
Let $D$ be a bounded, simply connected domain, take $\lambda\in(0,1/2]$ and $16/9<\theta<2$. Let $\kappa\in(8/3,4]$ be such that $\lambda = (3\kappa - 8)(6-\kappa) / 4\kappa$. As $N\to\infty$, the collection of outer boundaries of all outermost clusters of $\mathcal{\tilde L}_N^{N^{\theta-2}}$ converges in distribution to $\text{\textnormal{CLE}}_\kappa$, with respect to $d_H^*$.
\end{corollary}

Note that since $\theta < 2$, $\mathcal{\tilde L}_N^{N^{\theta-2}}$ contains loops of time length, and hence also diameter, arbitrarily small as $N \to \infty$, so the number of loops in $\mathcal{\tilde L}_N^{N^{\theta-2}}$ diverges as $N\to\infty$. Theorem \ref{mainthmBLS} has an analogue for the random walk loop soup with killing and the massive Brownian loop soup as defined in \cite{Cam}; our proof extends to that case.

We conclude this section by giving an outline of the paper and explaining the structure of the proof of Theorem \ref{mainthmBLS}. The largest part of the proof is to show that, for large $N$, with high probability, for each large cluster $C$ of $\BLS$ there exists a cluster $\tilde C_N$ of $\mathcal{\tilde L}_N^{N^{\theta-2}}$ such that $d_H(\Ext C, \Ext \tilde C_N)$ is small. We will prove this fact in three steps.

\begin{figure}
	\begin{center}
		\includegraphics[scale=0.7]{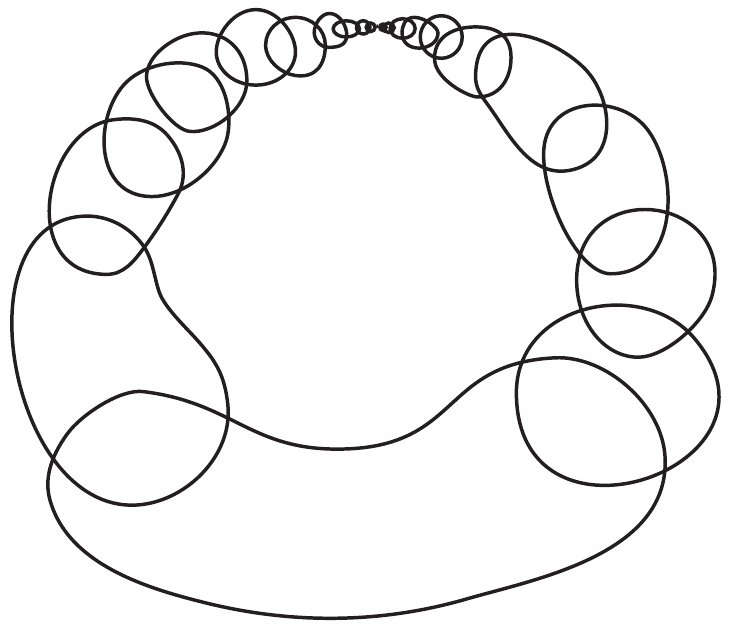}
	\end{center}
	\caption{A cluster whose exterior is not well-approximated by the exterior of any finite subcluster.}
	\label{fig:cluster}
\end{figure}

First, let $C$ be a large cluster of $\BLS$. We choose a finite subcluster $C'$ of $C$ such that $d_H(\Ext C, \Ext C')$ is small. A priori, it is not clear that such a finite subcluster exists -- see, e.g., Figure \ref{fig:cluster} which depicts a cluster containing two disjoint infinite chains of loops at Euclidean distance zero from each other. A proof that, almost surely, a finite subcluster with the desired property exists is given in Section \ref{sec:cluster}, using results from Section \ref{sec:topological}. The latter section contains a number of definitions and preliminary results used in the rest of the paper.

Second, we approximate the finite subcluster $C'$ by a finite subcluster $\tilde C_N'$ of $\mathcal{\tilde L}_N^{N^{\theta-2}}$. Here we use Corollary 5.4 of Lawler and Trujillo Ferreras \cite{LawTru}, which gives that, with probability tending to 1, there is a one-to-one correspondence between loops in $\mathcal{\tilde L}_N^{N^{\theta-2}}$ and loops in $\BLS^{N^{\theta-2}}$ such that corresponding loops are close. To prove that $d_H(\Ext C', \Ext \tilde C_N')$ is small, we need results from Section \ref{sec:topological} and the fact that a planar Brownian loop has no ``touchings'' in the sense of Definition \ref{deftouching} below. The latter result is proved in Section \ref{sec:notouchings}.

Third, we let $\tilde C_N$ be the full cluster of $\mathcal{\tilde L}_N^{N^{\theta-2}}$ that contains $\tilde C_N'$. In Section~\ref{sec:distance} we prove an estimate which implies that, with high probability, for non-intersecting loops in $\BLS^{N^{\theta-2}}$ the corresponding loops in $\mathcal{\tilde L}_N^{N^{\theta-2}}$ do not intersect. We deduce from this that, for distinct subclusters $\tilde C_{1,N}'$ and $\tilde C_{2,N}'$, the corresponding clusters $\tilde C_{1,N}$ and $\tilde C_{2,N}$ are distinct. We use this property to conclude that $d_H(\Ext C, \Ext \tilde C_N)$ is small.

\section{Preliminary results}\label{sec:topological}

In this section we give precise definitions and rigorous proofs of deterministic results which are important tools
in the proof of our main result. Let $\gamma_N$ be a sequence of curves converging uniformly to a curve $\gamma$, i.e.\ $\distinf(\gamma_N,\gamma) \to 0$ as $N \to \infty$, where
\[
\distinf(\gamma,\gamma') = \sup_{s \in [0,1]} |\gamma(st_{\gamma})-\gamma'(st_{\gamma'})| + |t_{\gamma} -t_{\gamma'}|.
\]
The distance $\distinf$ is a natural distance on the space of curves mentioned in Section 5.1 of \cite{Law}. We will identify topological conditions that, imposed on $\gamma$ (and $\gamma_N$), will yield convergence in the Hausdorff distance of the exteriors, outer boundaries and hulls
of $\gamma_N$ to the corresponding sets defined for $\gamma$.
Note that, in general, uniform convergence of the curves does not imply convergence of any of these sets. We define a notion of touching (see Figure~\ref{fig:touching}) and prove that if $\gamma$ has no touchings then the desired convergence follows:

\begin{figure}
	\begin{center}
		\includegraphics[scale=0.7]{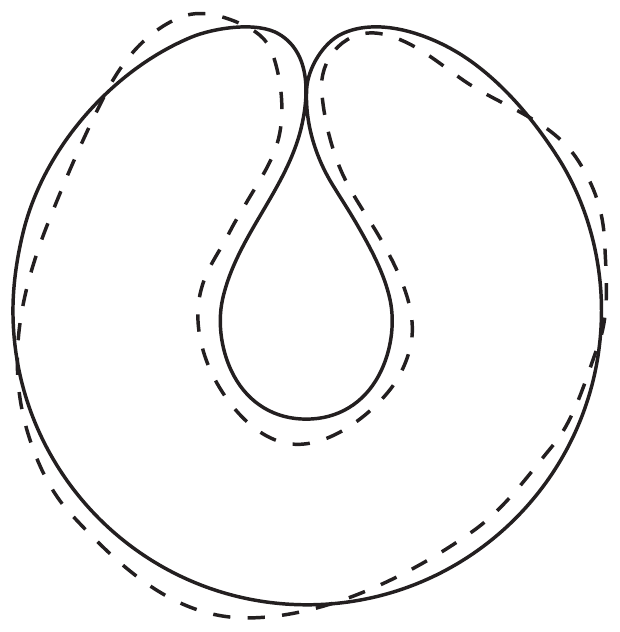}
	\end{center}
	\caption{A curve with a touching, and an approximating curve (dashed).}
	\label{fig:touching}
\end{figure}

\begin{definition}\label{deftouching}
We say that a curve $\gamma$ has a \emph{touching} $(s,t)$ if $0 \leq s < t \leq t_{\gamma}$, $\gamma(s) = \gamma(t)$ and there exists $\delta > 0$ such that for all $\varepsilon \in (0,\delta)$, there exists a curve $\gamma'$ with $t_{\gamma}=t_{\gamma'}$, such that $d_{\infty}(\gamma,\gamma') < \varepsilon$ and $\gamma'[s^-,s^+] \cap \gamma'[t^-,t^+] = \emptyset$, where $(s^-,s^+)$ is the largest subinterval of $[0,t_{\gamma}]$ such that $s^-\leq s\leq s^+$ and $\gamma'(s^-,s^+)\subset B(\gamma(s);\delta)$, and $t^-,t^+$ are defined similarly using $t$ instead of $s$.
\end{definition}

\begin{theorem} \label{thm:curveconv}
Let $\gamma_N,\gamma$ be curves such that $\distinf(\gamma_N,\gamma) \to 0$ as $N\to \infty$, and $\gamma$ has no touchings. Then,
\[
\dist_H(\Ext \gamma_N , \Ext\gamma) \to 0, \
\dist_H( \obound \gamma_N , \obound \gamma) \to 0, \  \text{and}\
 \dist_H( \Hull \gamma_N , \Hull \gamma) \to 0.
\]
\end{theorem}

To prove the main result of this paper, we will also need to deal with similar convergence issues for sets defined by collections of curves. For two collections of curves $C,C'$ let
\begin{align*}
d^*_{\infty}(C,C') = & \inf \{ \delta>0 : \forall \gamma\in C \,\exists \gamma'\in C' \text{ such that } \dist_{\infty}(\gamma,\gamma') < \delta, \\
& \text{and } \forall \gamma'\in C' \,\exists \gamma\in C \text{ such that } \dist_{\infty}(\gamma,\gamma') < \delta \}.
\end{align*}
We will also need a modification of the notion of touching:

\begin{definition}\label{defmutualtouching}
Let $\gamma_1$ and $\gamma_2$ be curves.
We say that the pair $\gamma_1,\gamma_2$ has a \emph{mutual touching} $(s,t)$ if $0\leq s\leq t_{\gamma_1}$, $0\leq t\leq t_{\gamma_2}$, $\gamma_1(s)=\gamma_2(t)$
and there exists $\delta>0$ such that for all $\varepsilon\in(0,\delta)$, there exist curves $\gamma_1'$, $\gamma_2'$ with $t_{\gamma_1} =t_{\gamma_1'}$, $t_{\gamma_2} =t_{\gamma_2'}$, such that
$d_{\infty}(\gamma_1,\gamma_1')<\varepsilon$, $d_{\infty}(\gamma_2,\gamma_2')<\varepsilon$ and $\gamma_1'[s^-,s^+]\cap \gamma_2'[t^-,t^+] =\emptyset$,
where $(s^-,s^+)$ is the largest subinterval of $[0,t_{\gamma_1}]$ such that $s^-\leq s\leq s^+$ and $\gamma_1'(s^-,s^+)\subset B(\gamma_1(s);\delta)$, and $t^-,t^+$ are
defined similarly using $\gamma_2$ and $t$, instead of $\gamma_1$ and $s$.
\end{definition}

\begin{definition}\label{deftouchingcollection}
We say that a collection of curves has a \emph{touching} if it contains a curve that has a touching or it contains a pair of distinct curves that have a mutual touching.
\end{definition}

The next result is an analog of Theorem~\ref{thm:curveconv}.
\begin{theorem} \label{thm:collectionconv}
Let $C_N,C$ be collections of curves such that $\distinf^*(C_N,C) \to 0$ as $N\to \infty$, and $C$ contains finitely many curves and $C$ has no touchings. Then,
\[
\dist_H(\Ext C_N , \Ext C) \to 0, \
\dist_H( \obound C_N , \obound C) \to 0, \ \text{and}\
 \dist_H( \Hull C_N , \Hull C) \to 0.
\]
\end{theorem}

The remainder of this section is devoted to proving Theorems~\ref{thm:curveconv} and~\ref{thm:collectionconv}.
We will first identify a general condition for the convergence of exteriors, outer boundaries and hulls in the setting of arbitrary bounded subsets of the plane. We will prove that if a curve does not have any touchings, then this condition is satisfied and hence Theorem~\ref{thm:curveconv} follows. At the end of the section, we will show how to obtain Theorem~\ref{thm:collectionconv} using similar arguments.

\begin{proposition} \label{prop:topological}
Let $A_N,A$ be bounded subsets of the plane such that $\dist_H(A_N, A) \to 0$ as $N\to \infty$. Suppose that for every $\delta >0$ there exists $N_0$ such that, for all $N>N_0$,
$\Ext  A_N \subset  (\Ext  A)^{\delta}$. Then,
\[
\dist_H(\Ext A_N , \Ext A) \to 0, \
\dist_H( \obound A_N , \obound A) \to 0, \  \text{and}\
 \dist_H( \Hull A_N , \Hull A) \to 0.
\]
 \end{proposition}

To prove Proposition \ref{prop:topological}, we will first prove that one of the inclusions required for the convergence of exteriors is always satisfied under the assumption that $\dist_H(A_N, A) \to 0$. For sets $A,A'$ let $d_{\mathcal{E}}(A,A')$ be the \emph{Euclidean distance} between $A$ and $A'$.

\begin{lemma} \label{lem:topological}
Let $A_N,A$ be bounded sets such that $\dist_H(A_N, A) \to 0$ as $N\to \infty$. Then, for every $\delta >0$, there exists~$N_0$ such that for all $N>N_0$, $\Ext A \subset (\Ext A_N)^{\delta}$.
\end{lemma}
\begin{proof}
Suppose that the desired inclusion does not hold. This means that there exists $\delta>0$ such that, after passing to a subsequence,
$\Ext A \not \subset (\Ext A_N)^{\delta}$ for all~$N$. This is equivalent to the existence of $x_N \in \Ext A$, such that $\dist_{\mathcal{E}}(x_N,\Ext A_N) \geq \delta$.
Since $\dist_H(A_N,A) \to 0$ and the sets are bounded, the sequence $x_N$ is bounded and we can assume
that $x_N\to x \in \overline{\Ext A}$ when $N \to \infty$.
It follows that for $N$ large enough, $\dist_{\mathcal{E}}(x,\Ext A_N) > \delta/2$ and hence $B(x;\delta/2)$ does not intersect $\Ext A_N$.
We will show that this leads to a contradiction. To this end, note that since $x\in \overline{\Ext A}$, there exists $y \in \Ext A$ such that $|x-y|<\delta/4$.
Furthermore, $\Ext A$ is an open connected subset of~$\IC$, and hence it is path connected. This means that there exists
a continuous path connecting $y$ with $\infty$ which stays within $\Ext A$. We denote by~$\wp$ its range in the complex
plane. Note that $\dist_{\mathcal{E}}(\wp,\overline{A})>0$. For $N$ sufficiently large, $\dist_H(\overline{A_N},\overline{A})<\dist_{\mathcal{E}}(\wp,\overline{A})$
and so $\overline{A_N}$ does not intersect $\wp$. This implies that $\overline{A_N}$ does not
disconnect $y$ from $\infty$. Hence, $y \in \Ext A_N$ and $B(x;\delta/2)$ intersects $\Ext A_N$ for $N$ large enough, which is a contradiction.
This completes the proof.
\end{proof}

\begin{lemma} \label{lem:helpfulinclusions}
Let $A,A'$ be bounded sets and let $\delta>0$. If $\dist_H(A,A')<\delta$ and $\Ext A \subset (\Ext A')^{\delta}$,
then $\obound A \subset (\obound A')^{2\delta}$ and $\Hull A' \subset (\Hull A)^{2\delta}$.
\end{lemma}

\begin{proof}
We start with the first inclusion.
From the assumption, it follows that $\overline{A} \subset (A')^{\delta}$ and $\overline{\Ext A} \subset (\Ext A')^{2\delta}$. Take $x \in \obound A$.
Since $\obound A \subset \overline{A} \subset (A')^{\delta} \subset (\Hull A')^{\delta} $, we have that $B (x; \delta) \cap \Hull A' \neq \emptyset$.
Since $\obound A \subset \overline{\Ext A} \subset  (\Ext A')^{2\delta}$, we have that $B (x; 2\delta) \cap \Ext A'\neq \emptyset$.
The ball $B (x; 2\delta)$ is connected and intersects both $\Ext A'$ and its complement $\Hull A'$. This implies that $B (x; 2\delta) \cap \obound A' \neq \emptyset$.
The choice of $x$ was arbitrary, and hence $\obound A \subset (\obound A')^{2\delta}$.

We are left with proving the second inclusion. From the assumption, it follows that
$\overline{A'} \subset A^{\delta}$ and $\Ext A \subset (\Ext A')^{\delta}$. Since $\obound A' \subset \overline{A'} \subset A^{\delta} \subset (\Hull A)^{\delta}$,
we have that $(\obound A')^{\delta} \subset (\Hull A)^{2\delta}$.
Since $\Ext A \subset (\Ext A')^{\delta} = \Ext A' \cup (\obound A')^{\delta}$, by taking complements we have
that $\Hull A' \setminus (\obound A')^{\delta} \subset \Hull A \subset (\Hull A)^{2 \delta}$.
By taking the union with $(\obound A')^{\delta}$, we obtain that $\Hull A' \subset (\Hull A)^{2\delta}$.
\end{proof}

\begin{proof}[Proof of Proposition~\ref{prop:topological}]
It follows from Lemmas~\ref{lem:topological} and~\ref{lem:helpfulinclusions}.
\end{proof}

\begin{remark} \label{remark:topological}
In the proof of Theorem~\ref{mainthmBLS}, we will use equivalent formulations of Theorem~\ref{thm:collectionconv} and Lemma~\ref{lem:topological} in terms of metric rather than
sequential convergence. The equivalent formulation of Lemma~\ref{lem:topological} is as follows: For any bounded set $A$ and $\delta>0$, there exists $\varepsilon >0$ such that
if $\dist_H(A,A')<\varepsilon$, then $\Ext A \subset (\Ext A')^{\delta}$. The equivalent formulation of Theorem~\ref{thm:collectionconv} is similar.
\end{remark}

Without loss of generality, from now till the end of this section, we assume that all curves have time length $1$
(this can always be achieved by a linear time change).

\begin{definition}
We say that $s,t \in [0,1]$ are \emph{$\delta$-connected} in a curve $\gamma$ if there exists an open ball $B$ of diameter $\delta$ such that $\gamma(s)$ and $\gamma(t)$
are connected in $\gamma \cap B$.
\end{definition}

\begin{lemma} \label{lem:touchingconv}
Let $\gamma_N,\gamma$ be curves such that $\distinf(\gamma_N,\gamma)\to 0$ as $N \to \infty$, and $\gamma$ has no touchings. Then for any $\delta>0$ and $s,t$ which are $\delta$-connected in $\gamma$, there exists $N_0$ such that $s,t$ are $4\delta$-connected in $\gamma_N$ for all $N>N_0$.
\end{lemma}
\begin{proof}
Fix $\delta >0$. If the diameter of $\gamma$ is at most $\delta$, then it is enough to take~$N_0$ such that $\distinf(\gamma_N,\gamma) < \delta$ for $N>N_0$.

Otherwise, let $s,t\in[0,1]$ be $\delta$-connected in $\gamma$ and let $x$ be such that $\gamma(s)$ and $\gamma(t)$ are in the same connected component of $\gamma \cap B(x;\delta/2)$.
We say that $I=[a,b] \subset [0,1]$ defines an \emph{excursion} of $\gamma$
from $\partial B(x;\delta)$ to $B(x;\delta/2)$ if $I$ is a maximal interval satisfying
\[
 \gamma(a,b) \subset B(x;\delta) \quad  \text{and} \quad  \gamma(a,b) \cap B(x;\delta/2) \neq \emptyset.
\]
Note that if $[a,b]$ defines an excursion, then the diameter of $\gamma[a,b]$ is at least $\delta/2$. Since $\gamma$ is uniformly continuous, it follows that there are only finitely many excursions.
Let $I_i=[a_i,b_i]$, $i=1,2,\ldots,k$, be the intervals which define them.

It follows that $\gamma \cap B(x;\delta/2) \subset \bigcup_{i=1}^k \gamma[ I_i]$, and hence $\gamma(s)$ and $\gamma(t)$ are in the same connected component
of $\bigcup_{i=1}^k \gamma[ I_i]$. If $s,t \in I_i$ for some $i$, then it is enough to take $N_0$ such that $\distinf(\gamma_N, \gamma) < \delta$ for $N>N_0$, and the claim of the lemma follows. Otherwise, using the fact that $\gamma[I_i]$ are closed, connected sets, one can reorder the intervals in such a way that
$s \in I_1$, $t \in I_l$, and $\gamma[I_i] \cap \gamma[I_{i+1}] \neq \emptyset$ for $i =1,\ldots, l-1$. Let $(s_i,t_i)$ be such that $s_i \in I_i$, $t_i \in I_{i+1}$, and $\gamma(s_i)=\gamma(t_i)=z_i$.
Since $(s_i,t_i)$ is not a touching, we can find $\varepsilon_i \in (0, \delta)$ such that
$\gamma'(s_i)$ is connected to $\gamma'(t_i)$ in $\gamma' \cap B(z_i;\delta)$ for all $\gamma'$ with $\dist(\gamma,\gamma')<\varepsilon_i$.
Hence, if $N_0$ is such that $\dist(\gamma_N,\gamma) < \min\{ \varepsilon, \delta\}$ for $N>N_0$, where $\varepsilon =\min_i \varepsilon_i$,
then $\gamma_N(s)$ and $\gamma_N(t)$ are connected in $\bigcup_{i=1}^l \gamma_N[I_i] \cup( \gamma_N\cap \bigcup_{i=1}^{l-1} B(z_i;\delta)) $,
and therefore also in $\gamma_N \cap B(x;2\delta)$.
\end{proof}

\begin{lemma} \label{boundarycurve}
If $\gamma$ is a curve, then there exists a loop whose range is $\obound \gamma$ and whose winding around each point
of $\Hull \gamma \setminus \obound \gamma$ is equal to $2\pi$.
\end{lemma}
\begin{proof}
Let $D'=\{ x \in \IC: |x| >1\}$. By the proof of Theorem 1.5(ii) of~\cite{BurLaw}, there exists a one-to-one conformal map $\varphi$ from $D'$ onto $\Ext \gamma$
which extends to a continuous function $\overline{\varphi} : \overline{D'} \to \overline{\Ext \gamma}$, and such that $\overline{\varphi}[\partial D'] = \obound \gamma$.
Let $\gamma_r(t)=\overline{\varphi}(e^{it2\pi}(1+r))$ for $t \in [0,1]$ and $r\geq0$. It follows that the range of $\gamma_0$ is $\obound \gamma$.
Moreover, since $\varphi$ is one-to-one, $\gamma_r$ is a simple curve for $r>0$ and hence its winding around every point of $\Hull \gamma \setminus \obound \gamma$ is equal to $2\pi$.
Since $\distinf(\gamma_0,\gamma_r) \to 0$ when $r \to 0$, the winding of  $\gamma_0$ around every point of $\Hull \gamma \setminus \obound \gamma$ is also equal to~$2\pi$.
\end{proof}

\begin{lemma} \label{lem:deltaconv}
Let $\gamma_N,\gamma$ be curves such that $\distinf(\gamma_N,\gamma)\to 0$ as $N \to \infty$. Suppose that for any $\delta>0$ and $s,t$ which are $\delta$-connected in $\gamma$, there exists $N_0$ such that
$s,t$ are $4\delta$-connected in $\gamma_N$ for all $N>N_0$. Then, for every $\delta >0$, there exists $N_0$ such that for all $N>N_0$,
$\Ext  \gamma_N \subset  (\Ext  \gamma)^{\delta}$.
\end{lemma}
\begin{proof}
Fix $\delta >0$. By Lemma~\ref{boundarycurve}, let $\gamma_0$ be a loop whose range is $\obound \gamma$ and
whose winding around each point of $\Hull \gamma \setminus \obound \gamma$ equals $2\pi$.  Let
\[
0=t_0 < t_1 < \ldots < t_l =1,
\]
be a sequence of times satisfying
\[
t_{i+1} = \inf \{ t \in [t_i, 1]: | \gamma_0(t)- \gamma_0(t_i)| = \delta/32 \big\} \quad \text{for} \quad  i = 0,\ldots, l-2,
\]
and $| \gamma_0(t)- \gamma_0(t_{l-1})| < \delta/32 $ for all $t\in [t_{l-1},1)$.
This is well defined, i.e.\ $l < \infty$, since $ \gamma_0$ is uniformly continuous. Note that $t_i$ and $t_{i+1}$ are $\delta/8$-connected in $ \gamma_0$.
For each~$t_i$, we choose a time $\tau_i$, such that $\gamma(\tau_i)= \gamma_0(t_i)$ and $\tau_l=\tau_0$. It follows that $\tau_i$ and $\tau_{i+1}$ are $\delta/8$-connected in $\gamma$.
Let $N_i$ be so large that $\tau_i$ and $\tau_{i+1}$ are $\delta/2$-connected in $\gamma_N$ for all $N> N_i$, and let $M=\max_i N_i$.
The existence of such $N_i$ is guaranteed by the assumption of the lemma.

Let $M' > M$ be such that $\distinf (\gamma_N,\gamma) < \delta /16$ for all $N>M'$. Take $N>M'$. We will show that $\Ext  \gamma_N \subset  (\Ext  \gamma)^{\delta}$.
Suppose by contradiction, that $x \in  \Ext \gamma_N \cap (\IC \setminus (\Ext  \gamma)^{\delta})=  \Ext \gamma_N \cap   (\Hull \gamma \setminus  (\obound \gamma)^{\delta}) $. Since $\Ext \gamma_N$ is open and connected, it is path connected
and there exists a continuous path~$\wp$ connecting $x$ with $\infty$ and such that $\wp \subset \Ext \gamma_N$.

We will construct a loop $\gamma^*$ which is contained in $ \IC \setminus \wp$, and which disconnects~$x$
from~$\infty$. This will yield a contradiction. By the definition of~$M$, for $i =0,\ldots, l-1$, there exists an open ball $B_i$ of diameter
$\delta/2$, such that $\gamma_N(\tau_i)$ and $\gamma_N(\tau_{i+1})$ are connected in $\gamma_N\cap B_i$, and hence also in $B_i \setminus \wp$.
Since the connected components of $B_i \setminus \wp$ are open, they are path connected and there exists a curve $\gamma^*_i$ which starts at  $\gamma_N(\tau_i)$, ends at $\gamma_N(\tau_{i+1})$,
and is contained in $B_i \setminus \wp$.
By concatenating these curves, we construct the loop~$\gamma^*$, i.e.\
\[
\gamma^*(t)= \gamma^*_i\Big(\frac{t-t_i}{t_{i+1}-t_i}\Big)\quad \text{for} \quad t \in [t_i,t_{i+1}], \quad i=0,\ldots, l-1.
\]

By construction, $\gamma^* \subset \IC \setminus \wp$.
We will now show that $\gamma^*$ disconnects $x$ from $\infty$ by proving that its winding around $x$ equals $2\pi$.
By the definition of $t_{i+1}$, $\gamma_0(t_i,t_{i+1}) \subset B( \gamma_0(t_i); \delta/16)$. Since $\distinf(\gamma_N,\gamma)<\delta/16$ and $\gamma_0(t_i)=\gamma(\tau_i)$, it follows that
$\gamma_0(t_i,t_{i+1})\subset B(\gamma_N(\tau_i);\delta/8)$. By the definition of $\gamma^*_i$, $\gamma^*_i \subset B_i \subset B(\gamma_N(\tau_i); \delta/2)$.
Combining these two facts, we conclude that $\distinf( \gamma_0, \gamma^*) <5 \delta/8$. Since the winding of $ \gamma_0$ around every point of
$\Hull \gamma \setminus \obound \gamma$ is equal to $2\pi$, and since $x \in \Hull\gamma$ and $\dist_{\mathcal{E}}(x, \gamma_0) \geq \delta$, the winding of $\gamma^*$
around $x$ is also equal to $2\pi$. This means that $\gamma^*$ disconnects $x$ from~$\infty$, and hence $\wp \cap \gamma^* \neq \emptyset$, which is a contradiction.
\end{proof}

\begin{proof}[Proof of Theorem~\ref{thm:curveconv}]
It is enough to use Proposition~\ref{prop:topological}, Lemma~\ref{lem:touchingconv} and Lemma~\ref{lem:deltaconv}.
\end{proof}

\begin{proof}[Proof of Theorem~\ref{thm:collectionconv}]

The proof follows similar steps as the proof of Theorem~\ref{thm:curveconv}.
To adapt Lemma~\ref{lem:touchingconv} to the setting of collections of curves, it is enough to notice that a finite collection of nontrivial curves, when intersected with a ball of sufficiently small radius, looks like a single curve intersected with the ball. To generalize Lemma~\ref{lem:deltaconv}, it suffices to notice that the outer boundary of each connected component of $C$ is given by a curve as in Lemma~\ref{boundarycurve}. \qedhere
\end{proof}

\section{Finite approximation of a Brownian loop soup cluster}\label{sec:cluster}

Let $\mathcal{L}$ be a Brownian loop soup with intensity $\lambda\in(0,1/2]$ in a bounded, simply connected domain $D$. The following theorem is the main result of this section.

\begin{theorem} \label{finiteapproximation}
Almost surely, for any cluster $C$ of~$\BLS$, there exists a sequence of finite subclusters $C_N$ of $C$ such that as $N\to \infty$,
\[
\dist_H(\Ext C_N , \Ext C) \to 0, \
\dist_H( \obound C_N , \obound C) \to 0, \  \text{and}\
 \dist_H( \Hull C_N , \Hull C) \to 0.
\]
\end{theorem}

We will need the following result.
\begin{lemma} \label{convergingloops}
 Almost surely, for each cluster $C$ of $\BLS$, there exists a sequence of finite subclusters $C_N$ increasing to~$C$ (i.e.\ $C_N \subset C_{N+1}$ for all $N$ and $\bigcup_N C_N =C$), and a sequence of loops $\ell_N:[0,1] \to \IC$ converging uniformly to a loop
$\ell:[0,1] \to \IC$, such that the range of $\ell_N$ is equal to
$C_N$, and hence the range of $\ell$ is equal to~$\overline{C}$.
\end{lemma}
\begin{proof}
This follows from the proof of Lemma~9.7 in~\cite{SheWer}.
Note that in~\cite{SheWer}, a cluster $C$ is replaced by the collection of simple loops $\eta$ given by the outer boundaries of $\gamma \in C$.
However, the same argument works also for $C$ and the loops $\gamma$.
\end{proof}

To prove Theorem \ref{finiteapproximation}, we will show that the loops $\ell_N$, $\ell$ from Lemma~\ref{convergingloops} satisfy the conditions of Lemma~\ref{lem:deltaconv}. Then, using Proposition \ref{prop:topological} and Lemma~\ref{lem:deltaconv}, we obtain Theorem \ref{finiteapproximation}. We will first prove some necessary lemmas.

\begin{lemma} \label{nonzerodistance}
Almost surely, for all $\gamma \in \BLS$ and all subclusters $C$ of $\BLS$ such that $\gamma$ does not intersect $C$, it holds that
$\dist_{\mathcal{E}}(\gamma, C) >0$.
\end{lemma}
\begin{proof}
Fix $k$ and let $\gamma_k$ be the loop in $\BLS$ with $k$-th largest diameter. Using an argument similar to that in Lemma~9.2 of \cite{SheWer},
one can prove that, conditionally on $\gamma_k$, the loops in $\BLS$ which do not intersect~$\gamma_k$ are distributed like
$\BLS(D \setminus \gamma_k)$, i.e.\ a Brownian loop soup in $D\setminus \gamma_k$. Moreover,  $\BLS(D \setminus \gamma_k)$
consists of a countable collection of disjoint loop soups, one for each connected component of  $D \setminus \gamma_k$. By conformal invariance, each
of these loop soups is distributed like a conformal image of a copy of $\BLS$. Hence, by Lemma~9.4 of~\cite{SheWer}, almost surely,
each cluster of $\BLS(D\setminus \gamma_k)$ is at positive distance from~$\gamma_k$. This implies that the unconditional probability that there exists a subcluster~$C$ such that $d_{\mathcal{E}}(\gamma_k,C) =0$ and $\gamma_k$ does not intersect $C$ is zero. Since~$k$ was arbitrary and
there are countably many loops in $\BLS$, the claim of the lemma follows.
\end{proof}

\begin{lemma} \label{rationalballs}
Almost surely, for all $x$ with rational coordinates and all rational $\delta >0$, no two clusters of the loop soup obtained by restricting $\BLS$ to $B(x;\delta)$ are at Euclidean distance zero from each other.
\end{lemma}
\begin{proof}
This follows from Lemma~9.4 of~\cite{SheWer}, the restriction property of the Brownian loop soup, conformal invariance and the fact that we consider a countable number of balls.
\end{proof}

\begin{lemma}\label{existencerho} Almost surely, for every $\delta>0$ there exists $t_0>0$ such that every subcluster of $\mathcal{L}$
with diameter larger than $\delta$ contains a loop of time length larger than $t_0$.
\end{lemma}
\begin{proof}
Let $\delta>0$ and suppose that for all $t_0>0$ there exists a subcluster of diameter larger than $\delta$ containing only loops of time length less than $t_0$.

Let $t_1 = 1$ and let $C_1$ be a subcluster of diameter larger than $\delta$ containing only loops of time length less than $t_1$.
By the definition of a subcluster there exists a finite chain of loops $C'_1$ which is a subcluster of $C_1$ and has diameter larger than~$\delta$.
Let $t_2 = \min \{ t_{\gamma} : \gamma \in C'_1 \}$, where $t_{\gamma}$ is the time length of $\gamma$. Let $C_2$ be a
subcluster of diameter larger than~$\delta$ containing only loops of time length less than $t_2$. By the definition of a subcluster there
exists a finite chain of loops $C'_2$ which is a subcluster of $C_2$ and has diameter larger than $\delta$. Note that by the construction
$\gamma_1 \neq \gamma_2$ for all $\gamma_1\in C'_1$, $\gamma_2\in C'_2$, i.e.\ the chains of loops $C'_1$ and $C'_2$ are disjoint as
collections of loops, i.e.\ $\gamma_1\neq \gamma_2$ for all $\gamma_1\in C_1',\gamma_2\in C_2'$. Iterating the construction gives infinitely
many chains of loops $C'_i$ which are disjoint as collections of loops and which have diameter larger than $\delta$.

For each chain of loops $C'_i$ take a point $z_i \in C'_i$, where $C'_i$ is viewed as a subset of the complex plane. Since the domain is bounded, the sequence $z_i$ has an accumulation point, say $z$.
Let $z'$ have rational coordinates and $\delta'$ be a rational number such that $|z-z'| < \delta/8$ and $|\delta-\delta'| < \delta/8$. The annulus centered at $z'$ with inner radius $\delta'/4$ and
outer radius $\delta'/2$ is crossed by infinitely many chains of loops which are disjoint as collections of loops. However, the latter event has probability~$0$ by Lemma~9.6 of \cite{SheWer} and its consequence, leading to a contradiction.
\end{proof}

\begin{proof}[Proof of Theorem \ref{finiteapproximation}]
We restrict our attention to the event of probability~$1$ such that the claims of Lemmas~\ref{convergingloops},~\ref{nonzerodistance},~\ref{rationalballs} and~\ref{existencerho} hold true, and such that there are only finitely many loops of diameter or time length larger than any positive threshold. Fix a realization of $\BLS$ and a cluster $C$ of $\BLS$. Take $C_N$, $\ell_N$ and $\ell$ defined for $C$ as in Lemma~\ref{convergingloops}. By Proposition~\ref{prop:topological} and Lemma~\ref{lem:deltaconv}, it is enough to prove that the sequence $\ell_N$ satisfies the condition that for all $\delta>0$ and $s,t\in[0,1]$ which are $\delta$-connected in $\ell$, there exists~$N_0$ such that $s,t$ are $4\delta$-connected in $\ell_N$ for all $N>N_0$.

To this end, take $\delta >0$ and $s,t$ such that $\ell(s)$ is connected to $\ell(t)$ in $\ell \cap B(x, \delta/2)$ for some $x$. Take $x'$ with rational coordinates and $\delta'$ rational such that $B(x;\delta/2) \subset B(x';\delta'/2)$ and $\overline{B(x';\delta')} \subset B(x;2\delta)$. If $C \subset B(x';\delta')$, then $\ell_N(s)$ is connected to $\ell_N(t)$ in $\ell_N\cap B(x;2\delta)$ for all $N$ and we are done. Hence, we can assume that
\begin{equation}\label{clusterintersectsboundary}
C \cap \partial B(x';\delta') \neq \emptyset.
\end{equation}

When intersected with $\overline{B(x';\delta')}$, each loop $\gamma \in C$ may split into multiple connected components.
We call each such component of $\gamma \cap \overline{B(x';\delta')}$ a \emph{piece} of $\gamma$. In particular if $\gamma \subset \overline{B(x';\delta')}$, then the only piece of $\gamma$ is the full loop~$\gamma$. The collection of all pieces we consider is given by $\{ \wp : \wp \text{ is a piece of } \gamma \text{ for some } \gamma\in C \}$. A \emph{chain of pieces} is a sequence of pieces such that each piece intersects the next piece in the sequence. Two pieces are in the same \emph{cluster of pieces} if they are connected via a finite chain of pieces. We identify a collection of pieces with the set in the plane given by the union of the pieces. Note that there are only finitely many pieces of diameter larger than any positive threshold, since the number of loops of diameter larger than any positive threshold is finite and each loop is uniformly continuous.

Let $C_1^*,C_2^*,\ldots$ be the clusters of pieces such that
\begin{equation}\label{propertyclusters}
C_i^* \cap B(x';\delta'/2) \neq \emptyset \text{ and } C_i^* \cap \partial B(x';\delta') \neq \emptyset.
\end{equation}
We will see later in the proof that the number of such clusters of pieces is finite, but we do not need this fact yet. We now prove that
\begin{equation}\label{clustersdisjoint}
\overline{C_i^*} \cap \overline{C_j^*} \cap B(x';\delta'/2) = \emptyset \text{ for all } i\neq j.
\end{equation}
To this end, suppose that \eqref{clustersdisjoint} is false and let $z \in \overline{C^*_i} \cap  \overline{C^*_j} \cap B(x';\delta'/2)$ for some $i\neq j$.

First assume that $z \in C^*_i$. Then, by the definition of clusters of pieces, $z \notin C^*_j$. It follows that $C^*_j$ contains a chain of infinitely many different pieces which has $z$ as an accumulation point. Since there are only finitely many pieces of diameter larger than any positive threshold, the diameters of the pieces in this chain approach $0$. Since $\dist_{\mathcal{E}}(z, \partial B(x';\delta')) > \delta'/2$, the pieces become full loops at some point in the chain. Let $\gamma\in C$ be such that $z\in\gamma$. It follows that there exists a subcluster of loops of $C$, which does not contain $\gamma$ and has $z$ as an accumulation point. This contradicts the claim of Lemma~\ref{nonzerodistance} and therefore it cannot be the case that $z\in C^*_i$.

Second assume that $z \notin C^*_i$ and $z \notin C^*_j$. By the same argument as in the previous paragraph, there exist two chains of loops of $C$ which are disjoint, contained in $B(x';\delta')$ and both of which have $z$ as an accumulation point.
These two chains belong to two different clusters of $\BLS$ restricted to $B(x';\delta')$. Since $x'$ and $\delta'$ are rational, this contradicts the claim of Lemma~\ref{rationalballs}, and hence it cannot be the case that $z \notin C^*_i$ and $z \notin C^*_j$. This completes the proof of \eqref{clustersdisjoint}.

We now define a particular collection of pieces $P$. By Lemma \ref{existencerho}, let $t_0>0$ be such that every subcluster of $\BLS$ of diameter larger than $\delta'/4$ contains a loop of time length larger than $t_0$. Let $P$ be the collection of pieces which have diameter larger than $\delta'/4$ or are full loops of time length larger than $t_0$. Note that $P$ is finite. Each chain of pieces which intersects both $B(x';\delta'/2)$ and $\partial B(x';\delta')$, contains a piece of diameter larger than $\delta'/4$ intersecting $\partial B(x';\delta')$ or contains a chain of full loops which intersects both $B(x';\delta'/2)$ and $\partial B(x';3\delta'/4)$. In the latter case it contains a subcluster of $\BLS$ of diameter larger than $\delta'/4$ and therefore a full loop of time length larger than $t_0$. Hence, each chain of pieces which intersects both $B(x';\delta'/2)$ and $\partial B(x';\delta')$ contains an element of $P$. Since $P$ is finite, it follows that the number of clusters of pieces $C_i^*$ satisfying \eqref{propertyclusters} is finite.

Since the range of $\ell$ is $\overline{C}$ and the number of clusters of pieces $C_i^*$ is finite,
\begin{align}
& \ell \cap B(x';\delta'/2) = \overline{C} \cap B(x';\delta'/2) \nonumber \\
&\textstyle = \overline{\bigcup_i C_i^*} \cap B(x';\delta'/2) = \bigcup_i \overline{C_i^*} \cap B(x';\delta'/2).\label{splitell}
\end{align}
By \eqref{clustersdisjoint}, \eqref{splitell} and the fact that $\ell(s)$ is connected to $\ell(t)$ in $\ell\cap B(x';\delta'/2)$,
\begin{equation}\label{ellincluster}
\ell(s),\ell(t)\in \overline{C_i^*} \cap B(x';\delta'/2),
\end{equation}
for some $i$. From now on see also Figure \ref{fig:pieces}.

\begin{figure}
	\begin{center}
		\includegraphics[scale=0.9]{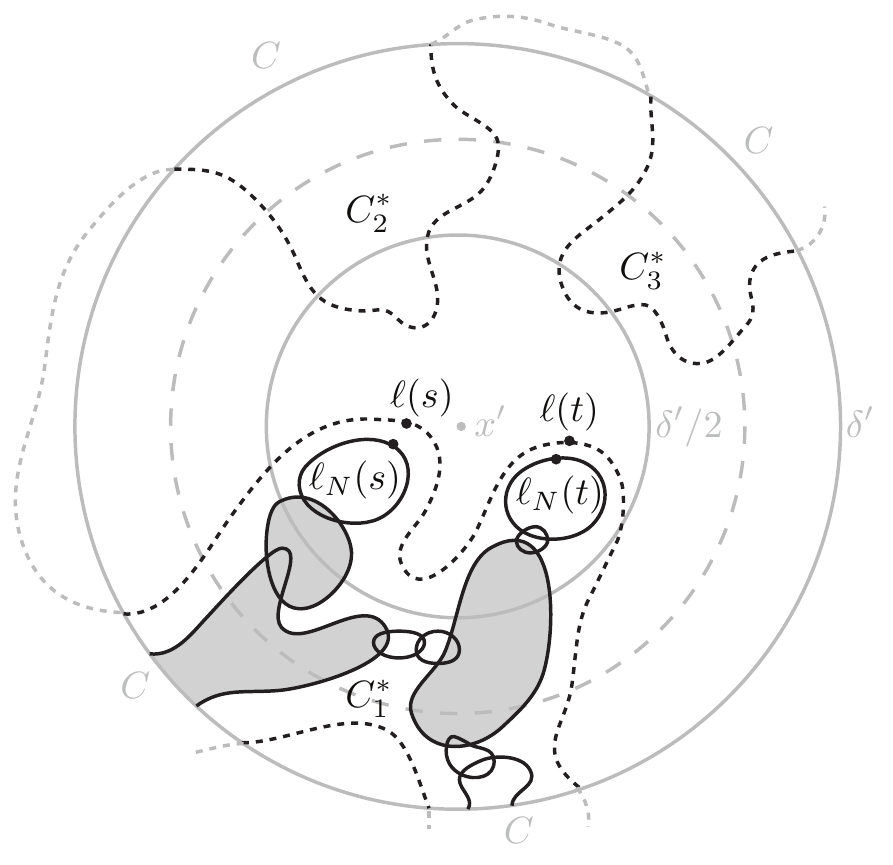}
	\end{center}
	\caption{Illustration of the last part of the proof of Theorem \ref{finiteapproximation} with $C_i^*=C_1^*$. The pieces drawn with solid lines form the set $C_i^*\cap \ell_N$. The shaded pieces represent the set $C_i^*\cap P$.}
	\label{fig:pieces}
\end{figure}

Let $\varepsilon$ be the Euclidean distance between $\{ \ell(s), \ell(t)\}$ and $\partial B(x';\delta'/2) \cup  \bigcup_{j\neq i}\overline{C^*_j}$. By \eqref{clustersdisjoint} and \eqref{ellincluster}, $\varepsilon>0$. Let $M$ be such that $\distinf(\ell_N,\ell)< \varepsilon$ and $\ell_N \cap \partial{B(x';\delta')} \neq \emptyset$ for $N>M$. The latter can be achieved by \eqref{clusterintersectsboundary}. Let $N>M$. By the definitions of $\varepsilon$ and $M$, we have that $\ell_N(s),\ell_N(t) \in  B(x';\delta'/2) $ and $\ell_N(s),\ell_N(t) \notin C^*_j$ for $j\neq i$. It follows that
\[
\ell_N(s),\ell_N(t) \in C_i^* \cap B(x';\delta'/2).
\]
Since $\ell_N$ is a finite subcluster of $C$, it also follows that there are finite chains of pieces $G^*_N(s), G^*_N(t)\subset C_i^*\cap \ell_N$ (not necessarily distinct) which connect $\ell_N(s),\ell_N(t)$, respectively, to $\partial B(x';\delta')$.

Since $G^*_N(s), G^*_N(t)$ intersect both $B(x';\delta'/2)$ and $\partial B(x';\delta')$, we have that $G_N^*(s),G_N^*(t)$ both contain an element of $P$. Moreover, $P$ is finite, any two elements of $C_i^*$ are connected via a finite chain of pieces and $\ell_N$ ($=C_N$) increases to the full cluster $C$. Hence, all elements of $C_i^*\cap P$ are connected to each other in $C_i^*\cap \ell_N$ for $N$ sufficiently large. It follows that $G^*_N(s)$ is connected to $G^*_N(t)$ in $C_i^*\cap\ell_N$ for $N$ sufficiently large. Hence, $\ell_N(s)$ is connected to $\ell_N(t)$ in $\ell_N \cap \overline{B(x';\delta')}$ for $N$ sufficiently large. This implies that $s,t$ are $4\delta$-connected in $\ell_N$ for $N$ sufficiently large.
\end{proof}

\section{No touchings}\label{sec:notouchings}

Recall the definitions of touching, Definitions \ref{deftouching}, \ref{defmutualtouching} and \ref{deftouchingcollection}. In this section we prove the following:

\begin{theorem}\label{thmbrownianmotion}
Let $B_t$ be a planar Brownian motion. Almost surely, $B_t$, $0 \leq t \leq 1$, has no touchings.
\end{theorem}

\begin{corollary}\label{corbrownianloop} ~ 
\begin{itemize}
\item[(i)] Let $B^{\Loop}_t$ be a planar Brownian loop with time length 1. Almost surely, $B^{\Loop}_t$, $0 \leq t \leq 1$, has no touchings.
\item[(ii)] Let $\mathcal{L}$ be a Brownian loop soup with intensity $\lambda\in(0,\infty)$ in a bounded, simply connected domain $D$. Almost surely, $\mathcal{L}$ has no touchings.
\end{itemize}
\end{corollary}

We start by giving a sketch of the proof of Theorem \ref{thmbrownianmotion}. Note that ruling out isolated touchings can be done using the fact that the intersection exponent $\zeta(2,2)$ is larger than 2 (see \cite{LawSchWer}). However, also more complicated situations like accumulations of touchings can occur. Therefore, we proceed as follows. We define excursions of the planar Brownian motion $B$ from the boundary of a disk which stay in the disk. Each of these excursions has, up to a rescaling in space and time, the same law as a process $W$ which we define below. We show that the process $W$ possesses a particular property, see Lemma \ref{Wnotouching} below. If $B$ had a touching, it would follow that the excursions of $B$ would have a behavior that is incompatible with this particular property of the process $W$.

As a corollary to Theorem \ref{thmbrownianmotion}, Corollary \ref{corbrownianloop} and Theorem \ref{thm:curveconv}, we obtain the following result. It is a natural result, but we could not find a version of this result in the literature and therefore we include it here.

\begin{corollary}\label{maincorBM}
Let $S_t$, $t\in\{0,1,2,\ldots\}$, be a simple random walk on the square lattice $\IZ^2$, with $S_0 = 0$, and define $S_t$ for non-integer times $t$ by linear interpolation.  
\begin{itemize}
\item[(i)] Let $B_t$ be a planar Brownian motion started at 0. As $N\to\infty$, the outer boundary of $(N^{-1} S_{2 N^2 t}, 0\leq t\leq 1)$ converges in distribution to the outer boundary of $(B_t, 0\leq t\leq 1)$, with respect to $d_H$.
\item[(ii)] Let $B_t^{\Loop}$ be a planar Brownian loop of time length 1 started at 0. As $N\to\infty$, the outer boundary of $(N^{-1} S_{2 N^2 t}, 0\leq t\leq 1)$, conditional on $\{S_{2 N^2} = 0\}$, converges in distribution to the outer boundary of $(B_t^{\Loop}, 0\leq t\leq 1)$, with respect to $d_H$.
\end{itemize}
\end{corollary}

To define the process $W$ mentioned above, we recall some facts about the three-dimensional Bessel process and its relation with Brownian motion, see e.g.\ Lemma 1 of~\cite{BurWer} and the references therein. The three-dimensional Bessel process can be defined as the modulus of a three-dimensional Brownian motion.

\begin{lemma}\label{bessel}
Let $X_t$ be a one-dimensional Brownian motion starting at 0 and $Y_t$ a three-dimensional Bessel process starting at 0. Let $0 < a < a'$ and define $\tau = T_a(X) = \inf \{ t\geq 0: X_t = a\}$, $\tau' = T_{a'}(X)$, $\sigma = \sup \{ t < \tau : X_t = 0 \}$, $\rho = T_a(Y)$ and $\rho' = T_{a'}(Y)$. Then,
\begin{itemize}
\item[(i)] the two processes $(X_{\sigma+u}, 0 \leq u \leq \tau - \sigma)$ and $(Y_u, 0 \leq u \leq \rho)$ have the same law,
\item[(ii)] the process $(Y_{\rho+u}, 0 \leq u \leq \rho' - \rho)$ has the same law as the process $(X_{\tau+u}, 0\leq u \leq \tau' - \tau)$ conditional on $\{\forall u \in [0,\tau'-\tau], X_{\tau+u} \neq 0\}$.
\end{itemize}
\end{lemma}

Next we recall the skew-product representation of planar Brownian motion, see e.g.\ Theorem 7.26 of \cite{MorPer}: For a planar Brownian motion $B_t$ starting at 1, there exist two independent one-dimensional Brownian motions $X_t^1$ and $X_t^2$ starting at 0 such that
\[
B_t = \exp(X_{H(t)}^1 + i X_{H(t)}^2),
\]
where
\[
H(t) = \inf \left\{ h \geq 0 : \int_0^h \exp(2 X_u^1) du > t \right\} = \int_0^t \frac{1}{|B_u|^2} du.
\]

We define the process $W_t$ as follows. Let $X_t$ be a one-dimensional Brownian motion starting according to some distribution on $[0,2\pi)$. Let $Y_t$ be a three-dimensional Bessel process starting at 0, independent of $X_t$. Define
\[
V_t = \exp( -Y_{H(t)} + i X_{H(t)}),
\]
where
\[
H(t) = \inf \left\{ h \geq 0 : \int_0^h \exp(-2 Y_u) du > t \right\}.
\]
Let $B_t$ be a planar Brownian motion starting at 0, independent of $X_t$ and $Y_t$, and define
\[
W_t = \left\{
\begin{array}{ll}
V_t, & 0\leq t\leq \tau_{\frac{1}{2}}, \\
V_{\tau_{\frac{1}{2}}} + B_{t-\tau_{\frac{1}{2}}},\quad & \tau_{\frac{1}{2}} < t \leq \tau,
\end{array}
\right.
\]
with
\begin{align*}
\tau_{\frac{1}{2}} &= \textstyle \inf\{ t>0 : |V_t| = \frac{1}{2} \}, \\
\tau &= \inf\{ t>\tau_{\frac{1}{2}} : |V_{\tau_{\frac{1}{2}}} + B_{t-\tau_{\frac{1}{2}}}| = 1 \}.
\end{align*}
Note that $W_t$ starts on the unit circle, stays in the unit disk and is stopped when it hits the unit circle again.

\begin{figure}
	\begin{center}
		\includegraphics[scale=0.9]{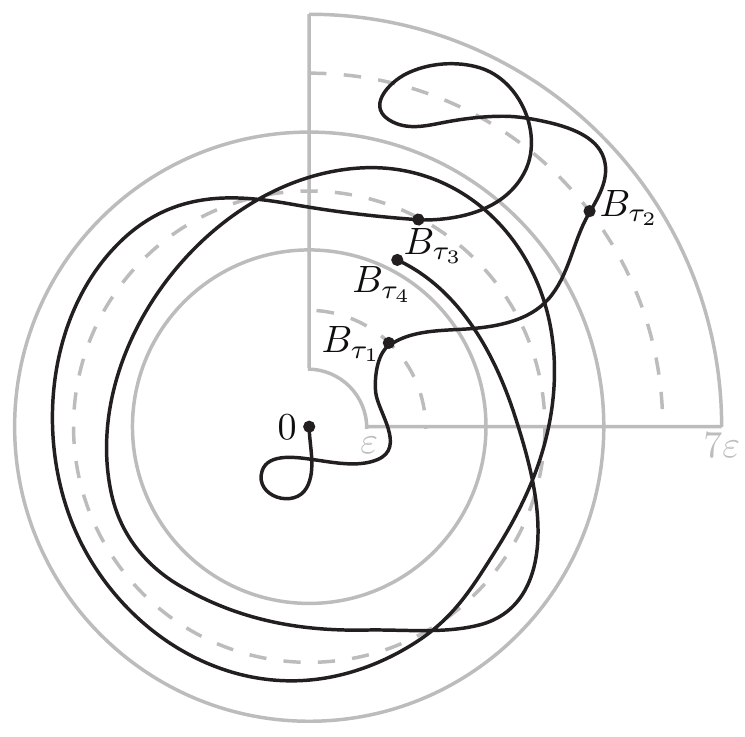}
	\end{center}
	\caption{The event $E_{\varepsilon}$.}
	\label{fig:winding}
\end{figure}

Next we derive the property of $W$ which we will use in the proof of Theorem \ref{thmbrownianmotion}. For this, we need the following property of planar Brownian motion:

\begin{lemma}\label{disconnection}
Let $B$ be a planar Brownian motion started at 0 and stopped when it hits the unit circle. Almost surely, there exists $\varepsilon>0$ such that for all curves $\gamma$ with $d_{\infty}(\gamma,B)<\varepsilon$ we have that $\gamma$ disconnects $\partial B(0;\varepsilon)$ from $\partial B(0;1)$.
\end{lemma}
\begin{proof}
We construct the event $E_{\varepsilon}$, for $0< \varepsilon \leq 1/7$, illustrated in Figure \ref{fig:winding}. Loosely speaking, $E_{\varepsilon}$ is the event that $B$ disconnects 0 from the unit circle in a strong sense, by crossing an annulus centered at 0 and winding around twice in this annulus. Let
\begin{align*}
\tau_1 &= \inf\{t\geq 0:|B_t| = 2\varepsilon\}, \quad \tau_2 = \inf\{t\geq 0:|B_t|=6\varepsilon\}, \\
\tau_3 &= \inf \{ t > \tau_2 : |B_t| = 4\varepsilon \}, \quad \tau_4 = \inf \{ t > \tau_3 : | \text{arg} (B_t / B_{\tau_3}) | = 4\pi \},
\end{align*}
where arg is the continuous determination of the angle. Let
\begin{align*}
A_1 &= \{ z\in\IC : \varepsilon < |z| < 7\varepsilon, | \text{arg} (z / B_{\tau_1}) | < \pi / 4 \}, \\
A_2 &= \{ z\in\IC : 3\varepsilon < |z| < 5\varepsilon \}.
\end{align*}
Define the event $E_{\varepsilon}$ by
\[
E_{\varepsilon} = \{ \tau_4 < \infty, B[\tau_1,\tau_3] \subset A_1, B[\tau_3,\tau_4] \subset A_2 \}.
\]

By construction, if $E_{\varepsilon}$ occurs then for all curves $\gamma$ with $d_{\infty}(\gamma,B)<\varepsilon$ we have that $\gamma$ disconnects $\partial B(0;2\varepsilon)$ from $\partial B(0;6\varepsilon)$. It remains to prove that almost surely $E_{\varepsilon}$ occurs for some $\varepsilon$. By scale invariance of Brownian motion, $\IP(E_{\varepsilon})$ does not depend on $\varepsilon$, and it is obvious that $\IP(E_{\varepsilon}) > 0$. Furthermore, the events $E_{1/7^n}$, $n \in \IN$, are independent. Hence almost surely $E_{\varepsilon}$ occurs for some $\varepsilon$.
\end{proof}

\begin{lemma}\label{Wnotouching}
Let $\gamma:[0,1]\to\IC$ be a curve with $|\gamma(0)|=|\gamma(1)|=1$ and $|\gamma(t)| < 1$ for all $t\in(0,1)$. Let $W$ denote the process defined above Lemma~\ref{disconnection} and assume that $W_0 \not\in \{\gamma(0),\gamma(1)\}$ a.s. Then the intersection of the following two events has probability 0:
\begin{itemize}
\item[(i)] $\gamma\cap W \neq \emptyset$,
\item[(ii)] for all $\varepsilon>0$ there exist curves $\gamma',\gamma''$ such that $d_{\infty}(\gamma,\gamma')<\varepsilon$, $d_{\infty}(W,\gamma'')<\varepsilon$ and $\gamma'\cap \gamma'' = \emptyset$.
\end{itemize}
\end{lemma}
\begin{proof}
The idea of the proof is as follows. We run the process $W_t$ till it hits $\partial B(0;a)$, where $a<1$ is close to 1. From that point the process is distributed as a conditioned Brownian motion. We run the Brownian motion till it hits the trace of the curve $\gamma$. From that point the Brownian motion winds around such that the event (ii) cannot occur, by Lemma \ref{disconnection}.

Let $T_a(W) = \inf\{t\geq 0: |W_t| = a\}$ and let $P$ be the law of $W_{T_a(W)}$. Let $B_t$ be a planar Brownian motion with starting point distributed according to the law $P$ and stopped when it hits the unit circle. Let $\tau = \inf\{t>0: |W_t|=1\}$. By Lemma \ref{bessel} and the skew-product representation, if $a\in(\frac{1}{2},1)$, the process
\[
(W_t, T_a(W) \leq t \leq \tau)
\]
has the same law as
\[
(B_t, 0\leq t \leq T_1(B)) \text{ conditional on } \{T_{1/2}(B) < T_1(B)\},
\]
where $T_1(B) = \inf\{t\geq 0:|B_t|=1\}$. Let $E_1,E_2$ be similar to the events (i) and (ii), respectively, from the statement of the lemma, but with $B$ instead of $W$, i.e.
\begin{align*}
E_1 &= \{\gamma\cap B \neq \emptyset\},\\
E_2 &= \{\text{for all } \varepsilon>0 \text{ there exist curves } \gamma',\gamma'' \text{ such that } d_{\infty}(\gamma,\gamma')<\varepsilon, \\
&\qquad d_{\infty}(B,\gamma'')<\varepsilon, \gamma'\cap\gamma'' = \emptyset \}.
\end{align*}
Let $T_{\gamma}(W) = \inf\{ t\geq 0 : W_t\in \gamma\}$ be the first time $W_t$ hits the trace of the curve $\gamma$.

The probability of the intersection of the events (i) and (ii) from the statement of the lemma is bounded above by
\begin{align}
&\IP(E_1\cap E_2 \mid T_{1/2}(B) < T_1(B)) + \IP(T_{\gamma}(W) \leq T_a(W)) \nonumber \\
&\leq \frac{\IP(E_2 \mid E_1) \IP(E_1)}{\IP(T_{1/2}(B) < T_1(B))} + \IP(T_{\gamma}(W) \leq T_a(W)). \label{intersectionbound}
\end{align}
The second term in \eqref{intersectionbound} converges to 0 as $a\to 1$, by the assumption that $W_0\not\in\{\gamma(0),\gamma(1)\}$ a.s. The first term in \eqref{intersectionbound} is equal to 0. This follows from the fact that
\begin{equation}\label{condprobzero}
\IP(E_2\mid E_1) = 0,
\end{equation}
which we prove below, using Lemma \ref{disconnection}.

To prove \eqref{condprobzero} note that $E_1 = \{ T_{\gamma}(B) \leq T_1(B)\}$, where $T_{\gamma}(B) = \inf\{t\geq 0 : B_t \in \gamma \}$. Define $\delta = 1-|B_{T_{\gamma}(B)}|$ and note that $\delta>0$ a.s. The time $T_{\gamma}(B)$ is a stopping time and hence, by the strong Markov property, $(B_t, t\geq T_{\gamma}(B))$ is a Brownian motion. Therefore, by translation and scale invariance, we can apply Lemma \ref{disconnection} to the process $(B_t, t\geq T_{\gamma}(B))$ stopped when it hits the boundary of the ball centered at $B_{T_{\gamma}(B)}$ with radius $\delta$. It follows that \eqref{condprobzero} holds.
\end{proof}

\begin{proof}[Proof of Theorem \ref{thmbrownianmotion}]
For $\delta_0>0$ we say that a curve $\gamma:[0,1]\to\IC$ has a \emph{$\delta_0$-touching} $(s,t)$ if $(s,t)$ is a touching and we can take $\delta=\delta_0$ in Definition \ref{deftouching}, and moreover $A \cap \partial B(\gamma(s);\delta_0) \neq \emptyset$ for all $A \in \{ \gamma[0,s), \gamma(s,t), \gamma(t,1] \}$. The last condition ensures that if $(s,t)$ is a $\delta_0$-touching then $\gamma$ makes excursions from $\partial B(\gamma(s);\delta_0)$ which visit $\gamma(s)$.

Since $B_t \neq B_0$ for all $t\in(0,1]$ a.s., we have that $(0,t)$ is not a touching for all $t\in(0,1]$ a.s. By time inversion, $B_1 - B_{1-u}$, $0 \leq u \leq 1$, is a planar Brownian motion and hence $(s,1)$ is not a touching for all $s\in[0,1)$ a.s. For every touching $(s,t)$ with $0<s<t<1$ there exists $\delta'>0$ such that for all $\delta\leq\delta'$ we have that $(s,t)$ is a $\delta$-touching a.s. (A touching $(s,t)$ that is not a $\delta$-touching for any $\delta>0$ could only exist if $B_u = B_0$ for all $u\in[0,s]$ or $B_u = B_1$ for all $u\in[t,1]$.) We prove that for every $\delta>0$ we have almost surely,
\begin{equation}\label{toprove}
B \text{ has no } \delta\text{-touchings } (s,t) \text{ with } 0<s<t<1.
\end{equation}
By letting $\delta\to 0$ it follows that $B$ has no touchings a.s.

To prove \eqref{toprove}, fix $\delta>0$ and let $z\in\IC$. We define excursions $W^n$, for $n\in\IN$, of the Brownian motion $B$ as follows. Let
\[
\tau_0 = \inf\{u\geq 0 : |B_u - z| = 2\delta/3\},
\]
and define for $n\geq 1$,
\begin{align*}
\sigma_n &= \inf \{ u>\tau_{n-1} : |B_u - z| = \delta/3\},\\
\rho_n &= \sup\{ u<\sigma_n : |B_u - z| = 2\delta/3\},\\
\tau_n &= \inf\{ u>\sigma_n : |B_u - z| = 2\delta/3\}.
\end{align*}
Note that $\rho_n<\sigma_n<\tau_n<\rho_{n+1}$ and that $\rho_n,\sigma_n,\tau_n$ may be infinite. The reason that we take $2\delta/3$ instead of $\delta$ is that we will consider $\delta$-touchings $(s,t)$ not only with $B_s=z$ but also with $|B_s-z|<\delta/3$. We define the excursion $W^n$ by
\[
W_u^n = B_u, \qquad \rho_n \leq u \leq \tau_n.
\]
Observe that $W^n$ has, up to a rescaling in space and time and a translation, the same law as the process $W$ defined above Lemma \ref{disconnection}. This follows from Lemma \ref{bessel}, the skew-product representation and Brownian scaling.

If $B$ has a $\delta$-touching $(s,t)$ with $|B_s-z|<\delta/3$, then there exist $m\neq n$ such that
\begin{itemize}
\item[(i)] $W^m \cap W^n \neq \emptyset$,
\item[(ii)] for all $\varepsilon>0$ there exist curves $\gamma^m,\gamma^n$ such that $d_{\infty}(\gamma^m,W^m)<\varepsilon$, $d_{\infty}(\gamma^n,W^n)<\varepsilon$ and $\gamma^m\cap\gamma^n = \emptyset$.
\end{itemize}
By Lemma \ref{Wnotouching}, with $W^m$ playing the role of $W$ and $W^n$ of $\gamma$, for each $m,n$ such that $m\neq n$ the intersection of the events (i) and (ii) has probability 0. Here we use the fact that $W^m_{\rho_m} \not\in \{W^n_{\rho_n},W^n_{\tau_n}\}$ a.s. Hence $B$ has no $\delta$-touchings $(s,t)$ with $|B_s-z|<\delta/3$ a.s. We can cover the plane with a countable number of balls of radius $\delta/3$ and hence $B$ has no $\delta$-touchings a.s.
\end{proof}

\begin{proof}[Proof of Corollary \ref{corbrownianloop}]
First we prove part (i). For any $u_0\in(0,1)$, the laws of the processes $B_u^{\Loop}$, $0\leq u\leq u_0$, and $B_u$, $0\leq u\leq u_0$, are mutually absolutely continuous, see e.g.\ Exercise 1.5(b) of \cite{MorPer}. Hence by Theorem \ref{thmbrownianmotion} the process $B_u^{\Loop}$, $0\leq u\leq 1$, has no touchings $(s,t)$ with $0\leq s < t \leq u_0$ a.s., for any $u_0\in(0,1)$. Taking a sequence of $u_0$ converging to 1, we have that $B_u^{\Loop}$, $0\leq u\leq 1$, has no touchings $(s,t)$ with $0\leq s < t < 1$ a.s. By time reversal, $B_1^{\Loop} - B_{1-u}^{\Loop}$, $0\leq u\leq 1$, is a planar Brownian loop. It follows that $B_u^{\Loop}$, $0\leq u\leq 1$, has no touchings $(s,1)$ with $s\in(0,1)$ a.s. By Lemma \ref{disconnection}, the time pair $(0,1)$ is not a touching a.s.

Second we prove part (ii). By Corollary \ref{corbrownianloop} and the fact that there are countably many loops in $\mathcal{L}$, we have that every loop in $\mathcal{L}$ has no touchings a.s. We prove that each pair of loops in $\mathcal{L}$ has no mutual touchings a.s. To this end, we discover the loops in $\mathcal{L}$ one by one in decreasing order of their diameter, similarly to the construction in Section 4.3 of \cite{NacWer}. Given a set of discovered loops $\gamma_1,\ldots,\gamma_{k-1}$, we prove that the next loop $\gamma_k$ and the already discovered loop $\gamma_i$ have no mutual touchings a.s., for each $i\in\{1,\ldots,k-1\}$ separately. Note that, conditional on $\gamma_1,\ldots,\gamma_{k-1}$, we can treat $\gamma_i$ as a deterministic loop, while $\gamma_k$ is a (random) planar Brownian loop. Therefore, to prove that $\gamma_k$ and $\gamma_i$ have no mutual touchings a.s., we can define excursions of $\gamma_i$ and $\gamma_k$ and apply Lemma \ref{Wnotouching} in a similar way as in the proof of Theorem \ref{thmbrownianmotion}. We omit the details.
\end{proof}

\section{Distance between Brownian loops}\label{sec:distance}

In this section we give two estimates, on the Euclidean distance between non-intersecting loops in the Brownian loop soup and on the overlap between intersecting loops in the Brownian loop soup. We will only use the first estimate in the proof of Theorem \ref{mainthmBLS}. As a corollary to the two estimates, we obtain a one-to-one correspondence between clusters composed of ``large'' loops from the random walk loop soup and clusters composed of ``large'' loops from the Brownian loop soup. This is an extension of Corollary 5.4 of \cite{LawTru}. For intersecting loops $\gamma_1,\gamma_2$ we define their \emph{overlap} by
\begin{align*}
\text{overlap}(\gamma_1,\gamma_2) = & \, 2 \sup \{ \varepsilon \geq 0 : \text{for all loops } \gamma_1',\gamma_2' \text{ such that } d_{\infty}(\gamma_1,\gamma_1')\leq\varepsilon, \\
& d_{\infty}(\gamma_2,\gamma_2') \leq \varepsilon, \text{ we have that } \gamma'_1 \cap \gamma'_2 \neq \emptyset \}.
\end{align*}

\begin{proposition}\label{correspondence1}
Let $\BLS$ be a Brownian loop soup with intensity $\lambda\in(0,\infty)$ in a bounded, simply connected domain $D$. Let $c>0$ and $16/9 < \theta < 2$. For all non-intersecting loops $\gamma,\gamma'\in \mathcal{L}$ of time length at least $N^{\theta-2}$ we have that $d_{\mathcal{E}}(\gamma,\gamma') \geq c N^{-1} \log N$, with probability tending to 1 as $N\to\infty$.
\end{proposition}

\begin{proposition}\label{correspondence2}
Let $\BLS$ be a Brownian loop soup with intensity $\lambda\in(0,\infty)$ in a bounded, simply connected domain $D$. Let $c>0$ and $\theta<2$ sufficiently close to 2. For all intersecting loops $\gamma,\gamma'\in \mathcal{L}$ of time length at least $N^{\theta-2}$ we have that $\text{\textnormal{overlap}}(\gamma,\gamma') \geq c N^{-1} \log N$, with probability tending to 1 as $N\to\infty$.
\end{proposition}

\begin{corollary}\label{lawlerext}
Let $D$ be a bounded, simply connected domain, take $\lambda\in(0,\infty)$ and $\theta<2$ sufficiently close to 2. Let $\mathcal{L}, \mathcal{L}^{N^{\theta-2}}, \mathcal{\tilde L}_N, \mathcal{\tilde L}_N^{N^{\theta-2}}$ be defined as in Section \ref{sec:definitions}. For every $N$ we can define $\mathcal{\tilde L}_N$ and $\mathcal{L}$ on the same probability space in such a way that the following holds with probability tending to 1 as $N\to\infty$. There is a one-to-one correspondence between the clusters of $\mathcal{\tilde L}_N^{N^{\theta-2}}$ and the clusters of $\mathcal{L}^{N^{\theta-2}}$ such that for corresponding clusters, $\tilde C\subset \mathcal{\tilde L}_N^{N^{\theta-2}}$ and $C\subset\mathcal{L}^{N^{\theta-2}}$, there is a one-to-one correspondence between the loops in $\tilde C$ and the loops in $C$ such that for corresponding loops, $\tilde\gamma\in \tilde C$ and $\gamma\in C$, we have that $d_{\infty}(\gamma,\tilde \gamma) \leq c N^{-1} \log N$, for some constant $c$ which does not depend on $N$.
\end{corollary}
\begin{proof}
Let $c$ be two times the constant in Corollary 5.4 of \cite{LawTru}. Combine this corollary and Propositions \ref{correspondence1} and \ref{correspondence2} with the $c$ in Propositions \ref{correspondence1} and \ref{correspondence2} equal to six times the constant in Corollary 5.4 of \cite{LawTru}.
\end{proof}

In Propositions \ref{correspondence1} and \ref{correspondence2} and Corollary \ref{lawlerext}, the probability tends to 1 as a power of $N$. This can be seen from the proofs. We will use Proposition \ref{correspondence1}, but we will not use Proposition \ref{correspondence2} in the proof of Theorem \ref{mainthmBLS}. Because of this, and because the proofs of Propositions \ref{correspondence1} and \ref{correspondence2} are based on similar techniques, we omit the proof of Proposition \ref{correspondence2}. To prove Proposition \ref{correspondence1}, we first prove two lemmas.

\begin{lemma}\label{diameterBrownianMotion}
Let $B$ be a planar Brownian motion and let $B^{\Loop,t_0}$ be a planar Brownian loop with time length $t_0$. There exist $c_1,c_2>0$ such that, for all $0 < \delta < \delta'$ and all $N \geq 1$,
\begin{align}
\IP(\diam B[0,N^{-\delta}] \leq N^{-\delta' / 2}) &\leq c_1 \exp(- c_2 N^{\delta' - \delta}), \label{diameter1} \\
\IP(\diam B^{\Loop,N^{-\delta}} \leq N^{-\delta' / 2}) &\leq c_1 \exp(- c_2 N^{\delta' - \delta}).\label{diameter2}
\end{align}
\end{lemma}
\begin{proof}
First we prove \eqref{diameter2}. By Brownian scaling,
\begin{align*}
&\IP(\diam B^{\Loop,N^{-\delta}} \leq N^{-\delta' / 2}) \nonumber\\
& = \IP(\diam B^{\Loop,1} \leq N^{-(\delta' - \delta) / 2}) \nonumber\\
& \textstyle \leq \IP(\sup_{t\in[0,1]} |X_t^{\Loop}| \leq N^{-(\delta' - \delta) / 2})^2,
\end{align*}
where $X^{\Loop}_t$ is a one-dimensional Brownian bridge starting at 0 with time length 1. The distribution of $\sup_{t\in[0,1]} |X_t^{\Loop}|$ is the asymptotic distribution of the (scaled) Kolmogorov-Smirnov statistic, and we can write, see e.g.\ Theorem 1 of \cite{Fel},
\begin{align}
& \textstyle \IP(\sup_{t\in[0,1]} |X_t^{\Loop}| \leq N^{-(\delta' - \delta) / 2}) \nonumber\\
&= \sqrt{2\pi} N^{(\delta'-\delta)/2} \textstyle \sum_{k=1}^{\infty} e^{-(2k-1)^2\pi^2 8^{-1} N^{\delta'-\delta}} \nonumber\\
&\leq \sqrt{2\pi} N^{(\delta'-\delta)/2} \textstyle \sum_{k=1}^{\infty} e^{-(2k-1)\pi^2 8^{-1} N^{\delta'-\delta}} \nonumber\\
&= \sqrt{2\pi} N^{(\delta'-\delta)/2} e^{\pi^2 8^{-1} N^{\delta'-\delta}} \textstyle \sum_{k=1}^{\infty} ( e^{-2\pi^2 8^{-1} N^{\delta'-\delta}} )^k \nonumber\\
&= \sqrt{2\pi} N^{(\delta'-\delta)/2} e^{-\pi^2 8^{-1} N^{\delta'-\delta}} (1-e^{-2\pi^2 8^{-1} N^{\delta'-\delta}})^{-1}\nonumber\\
&\leq c e^{-N^{\delta'-\delta}},\label{estimatesuploop}
\end{align}
for some constant $c$ and all $0<\delta<\delta'$ and all $N\geq 1$. This proves \eqref{diameter2}.

Next we prove \eqref{diameter1}. We can write $X_t^{\Loop} = X_t - t X_1$, where $X_t$ is a one-dimensional Brownian motion starting at 0. Hence
\begin{equation}\label{couplingbridge}
\sup_{t\in[0,1]} |X_t^{\Loop}| \leq \sup_{t\in[0,1]} |X_t| + |X_1| \leq  2 \sup_{t\in[0,1]} |X_t|.
\end{equation}
By Brownian scaling, \eqref{couplingbridge} and \eqref{estimatesuploop},
\begin{align*}
& \IP(\diam B[0,N^{-\delta}] \leq N^{-\delta'/2}) \\
&= \IP(\diam B[0,1] \leq N^{-(\delta'-\delta)/2}) \\
&\textstyle \leq \IP(\sup_{t\in [0,1]} |X_t| \leq N^{-(\delta'-\delta)/2})^2 \\
&\textstyle \leq \IP(\sup_{t\in [0,1]} |X_t^{\Loop}| \leq 2 N^{-(\delta'-\delta)/2})^2 \\
&\leq c^2 e^{-\frac{1}{2} N^{\delta'-\delta}}.
\end{align*}
This proves \eqref{diameter1}.
\end{proof}

\begin{lemma}\label{minimaldistance}
There exist $c_1,c_2>0$ such that the following holds. Let $c>0$ and $0<\delta<\delta'<2$. Let $\gamma$ be a (deterministic) loop with $\diam\gamma \geq N^{-\delta'/2}$. Let $B^{\Loop,t_0}$ be a planar Brownian loop starting at 0 of time length $t_0 \geq N^{-\delta}$. Then for all $N> 1$,
\begin{align*}
&\IP(0<d_{\mathcal{E}}(B^{\Loop,t_0},\gamma) \leq c N^{-1} \log N) \\
& \leq c_1 N^{-1/2+\delta'/4} (c \log N)^{1/2} + c_1 \exp(-c_2N^{\delta'-\delta})
\end{align*}
\end{lemma}
\begin{proof}
We use some ideas from the proof of Proposition 5.1 of \cite{LawTru}. By time reversal, we have
\begin{align}
& \IP(0 < d_{\mathcal{E}} ( B^{\Loop,t_0}, \gamma ) \leq c N^{-1} \log N) \nonumber\\
& \leq \textstyle 2 \IP(0 < d_{\mathcal{E}} ( B^{\Loop,t_0}[0,\frac{1}{2}t_0], \gamma ) \leq c N^{-1} \log N, B^{\Loop,t_0}[0,\frac{3}{4}t_0] \cap \gamma = \emptyset) \nonumber\\
& = 4 t_0 \lim_{\varepsilon \downarrow 0} \varepsilon^{-2} \textstyle \IP(0 < d_{\mathcal{E}} ( B[0,\frac{1}{2}t_0], \gamma ) \leq c N^{-1} \log N, B[0,\frac{3}{4}t_0] \cap \gamma = \emptyset, \label{star} \\
& \qquad \textstyle |B_{t_0}| \leq \varepsilon), \nonumber
\end{align}
where $B$ is a planar Brownian motion starting at 0. The equality \eqref{star} follows from the following relation between the law $\mu^{\sharp}_{0,t_0}$ of $(B^{\Loop,t_0}_t, 0\leq t\leq t_0)$ and the law $\mu_{0,t_0}$ of $(B_t, 0\leq t\leq t_0)$:
\[
\mu^{\sharp}_{0,t_0} = 2 t_0 \lim_{\varepsilon \downarrow 0} \varepsilon^{-2} \mu_{0,t_0} \I_{\{ |\gamma(t_0)| \leq \varepsilon \}},
\]
see Section 5.2 of \cite{Law} and Section 3.1.1 of \cite{LawWer}.

Next we bound the probability
\begin{equation}\label{comesClose}
\textstyle \IP(0 < d_{\mathcal{E}} ( B[0,\frac{1}{2}t_0], \gamma ) \leq c N^{-1} \log N, B[0,\frac{3}{4}t_0] \cap \gamma = \emptyset).
\end{equation}
If the event in \eqref{comesClose} occurs, then $B_t$ hits the $c N^{-1} \log N$ neighborhood of $\gamma$ before time $\frac{1}{2}t_0$, say at the point $x$. From that moment, in the next $\frac{1}{4}t_0$ time span, $B_t$ either stays within a ball containing $x$ (to be defined below) or exits this ball without touching $\gamma$. Hence, using the strong Markov property, \eqref{comesClose} is bounded above by
\begin{equation}\label{exitsBeforeTouching}
\sup_{\substack{x\in\IC, y\in\gamma \\ |x-y| \leq c N^{-1} \log N}}
\textstyle \IP(\tau^x_y > \frac{1}{4}t_0) + \IP(B^x[0,\tau^x_y] \cap \gamma = \emptyset),
\end{equation}
where $B^x$ is a planar Brownian motion starting at $x$ and $\tau^x_y$ is the exit time of $B^x$ from the ball $B(y; \frac{1}{4} N^{-\delta' / 2})$.

To bound the second term in \eqref{exitsBeforeTouching}, recall that $\diam\gamma \geq N^{-\delta' / 2}$, so $\gamma$ intersects both the center and the boundary of the ball $B(y; \frac{1}{4} N^{-\delta' / 2})$. Hence we can apply the Beurling estimate (see e.g.\ Theorem 3.76 of \cite{Law}) to obtain the following upper bound for the second term in \eqref{exitsBeforeTouching},
\begin{equation}\label{upperboundbeurling}
c_1 ( 4c N^{\delta' / 2} N^{-1} \log N )^{1/2},
\end{equation}
for some constant $c_1>1$ which in particular does not depend on the curve $\gamma$. The above reasoning to obtain the bound \eqref{upperboundbeurling} holds if $c N^{-1} \log N < \frac{1}{4} N^{-\delta'/2}$ and hence for large enough $N$. If $N$ is small then the bound \eqref{upperboundbeurling} is larger than 1 and holds trivially. To bound the first term in \eqref{exitsBeforeTouching} we use Lemma \ref{diameterBrownianMotion},
\begin{align*}
& \textstyle \IP(\tau^x_y > \frac{1}{4}t_0)
\leq \textstyle \IP(\tau^x_y > \frac{1}{4}N^{-\delta})
\leq \IP(\diam B[0, \frac{1}{4}N^{-\delta}] \leq \frac{1}{2} N^{-\delta' / 2}) \\
& \leq c_2 \exp( - c_3 N^{\delta' - \delta} ),
\end{align*}
for some constants $c_2,c_3>0$.

We have that
\begin{align}
& \textstyle \IP(|B_{t_0}| \leq \varepsilon \mid 0 < d_{\mathcal{E}} ( B[0,\frac{1}{2}t_0], \gamma ) \leq c N^{-1} \log N, B[0,\frac{3}{4}t_0] \cap \gamma = \emptyset) \nonumber\\
& \leq \sup_{x\in\IC} \IP(|B^x_{\frac{1}{4}t_0}| \leq \varepsilon) = \IP(|B_{\frac{1}{4}t_0}| \leq \varepsilon) \leq \frac{8}{\pi}\varepsilon^2 t_0^{-1}.\label{starstar}
\end{align}
The first inequality in \eqref{starstar} follows from the Markov property of Brownian motion. The equality in \eqref{starstar} follows from the fact that $B^x_{\frac{1}{4}t_0}$ is a two-dimensional Gaussian random vector centered at $x$. By combining \eqref{star}, the bound on \eqref{comesClose}, and \eqref{starstar}, we conclude that
\begin{align*}
& \IP(0 < d_{\mathcal{E}} ( B^{\Loop,t_0}, \gamma ) \leq c N^{-1} \log N) \\
& \leq \frac{32}{\pi} [ c_1 ( 4c N^{\delta' / 2} N^{-1} \log N )^{1/2} + c_2 \exp( - c_3 N^{\delta' - \delta} ) ].\qedhere
\end{align*}
\end{proof}

\begin{proof}[Proof of Proposition \ref{correspondence1}]
Let $2-\theta =: \delta < \delta' < 2$ and let $X_N$ be the number of loops in $\BLS$ of time length at least $N^{-\delta}$. First, we give an upper bound on $X_N$. Note that $X_N$ is stochastically less than the number of loops $\gamma$ in a Brownian loop soup in the full plane $\IC$ with $t_{\gamma} \geq N^{-\delta}$ and $\gamma(0)\in D$. The latter random variable has the Poisson distribution with mean
\[
\lambda \int_D \int_{N^{-\delta}}^{\infty} \frac{1}{2\pi t_0^2} dt_0 dA(z) = \lambda A(D) \frac{1}{2\pi} N^\delta,
\]
where $A$ denotes two-dimensional Lebesgue measure. By Chebyshev's inequality, $X_N \leq N^{\delta} \log N$ with probability tending to 1 as $N\to\infty$.

Second, we bound the probability that $\BLS$ contains loops of large time length with small diameter. By Lemma \ref{diameterBrownianMotion},
\begin{align}
& \IP(\exists \gamma \in \BLS, t_{\gamma} \geq N^{-\delta}, \diam\gamma < N^{-\delta' /2}) \nonumber\\
& \leq N^{\delta} \log N\,c_1 \exp(- c_2 N^{\delta' - \delta}) + \IP(X_N > N^{\delta} \log N),\label{corrbound1}
\end{align}
for some constants $c_1,c_2>0$. The expression \eqref{corrbound1} converges to 0 as $N\to\infty$.

Third, we prove the proposition. To this end, we discover the loops in $\BLS$ one by one in decreasing order of their time length, similarly to the construction in Section 4.3 of \cite{NacWer}. This exploration can be done in the following way. Let $\BLS_1,\BLS_2,\ldots$ be a sequence of independent Brownian loop soups with intensity $\lambda$ in $D$. From $\BLS_1$ take the loop $\gamma_1$ with the largest time length. From $\BLS_2$ take the loop $\gamma_2$ with the largest time length smaller than $t_{\gamma_1}$. Iterating this procedure yields a random collection of loops $\{\gamma_1,\gamma_2,\ldots\}$, which is such that $t_{\gamma_1}>t_{\gamma_2}>\cdots$ a.s. By properties of Poisson point processes, $\{\gamma_1,\gamma_2,\ldots\}$ is a Brownian loop soup with intensity $\lambda$ in $D$.

Given a set of discovered loops $\gamma_1,\ldots,\gamma_{k-1}$, we bound the probability that the next loop $\gamma_k$ comes close to $\gamma_i$ but does not intersect $\gamma_i$, for each $i \in\{1,\ldots,k-1\}$ separately. Note that, because of the conditioning, we can treat $\gamma_i$ as a deterministic loop, while $\gamma_k$ is random. Therefore, to obtain such a bound,  we can use Lemma \ref{minimaldistance} on the event that $t_{\gamma_k} \geq N^{-\delta}$ and $\diam\gamma_i \geq N^{-\delta' / 2}$. We use the first and second steps of this proof to bound the probability that $\BLS$ contains more than $N^{\delta}\log N$ loops of large time length, or loops of large time length with small diameter. Thus,
\begin{align}
& \IP(\exists \gamma,\gamma' \in \BLS, t_{\gamma},t_{\gamma'}\geq N^{-\delta}, 0 < d_{\mathcal{E}}(\gamma,\gamma') \leq c N^{-1} \log N) \nonumber \\
& \leq ( N^{\delta} \log N )^2 [ c_3 N^{-1/2 + \delta' / 4} (c \log N)^{1/2} + c_3 \exp(-c_4 N^{\delta'-\delta}) ] + \nonumber \\
& \qquad \IP(\{X_N > N^{\delta} \log N\} \cup \{\exists \gamma \in \BLS, t_{\gamma}\geq N^{-\delta}, \diam\gamma < N^{-\delta' /2}\}),\label{corrbound2}
\end{align}
for some constants $c_3,c_4>0$. If $\delta'< 2/9$, then \eqref{corrbound2} converges to 0 as $N\to\infty$.
\end{proof}

\section{Proof of main result}\label{sec:proof}

\begin{proof}[Proof of Theorem \ref{mainthmBLS}]
By Corollary 5.4 of \cite{LawTru}, for every $N$ we can define on the same probability space $\mathcal{\tilde L}_N$ and $\mathcal{L}$ such that the following holds with probability tending to 1 as $N\to\infty$: There is a one-to-one correspondence between the loops in $\mathcal{\tilde L}_N^{N^{\theta-2}}$ and the loops in $\mathcal{L}^{N^{\theta-2}}$ such that, if $\tilde\gamma \in \mathcal{\tilde L}_N^{N^{\theta-2}}$ and $\gamma\in \mathcal{L}^{N^{\theta-2}}$ are paired in this correspondence, then $d_{\infty}(\tilde\gamma,\gamma) < c N^{-1}\log N$, where $c$ is a constant which does not depend on $N$.

We prove that in the above coupling, for all $\delta,\alpha>0$ there exists $N_0$ such that for all $N\geq N_0$ the following holds with probability at least $1-\alpha$: For every outermost cluster $C$ of $\mathcal{L}$ there exists an outermost cluster $\tilde C_N$ of $\mathcal{\tilde L}_N^{N^{\theta-2}}$ such that
\begin{equation}\label{toprovemain2}
d_H(C,\tilde C_N) < \delta,\qquad d_H(\Ext C,\Ext \tilde C_N) < \delta,
\end{equation}
and for every outermost cluster $\tilde C_N$ of $\mathcal{\tilde L}_N^{N^{\theta-2}}$ there exists an outermost cluster $C$ of $\mathcal{L}$ such that \eqref{toprovemain2} holds. By Lemma \ref{lem:helpfulinclusions}, \eqref{toprovemain2} implies that $d_H(\partial_o C, \partial_o \tilde C_N) < 2\delta$ and $d_H(\Hull C,\Hull \tilde C_N) < 2\delta$. Also, \eqref{toprovemain2} implies that the Hausdorff distance between the carpet of $\BLS$ and the carpet of $\mathcal{\tilde L}_N^{N^{\theta-2}}$ is less than or equal to $\delta$. Hence this proves the theorem.

Fix $\delta,\alpha>0$. To simplify the presentation of the proof of \eqref{toprovemain2}, we will often use the phrase ``with high probability'', by which we mean with probability larger than a certain lower bound which is uniform in $N$. It is not difficult to check that we can choose these lower bounds in such a way that \eqref{toprovemain2} holds with probability at least $1-\alpha$.

First we define some constants. By Lemma 9.7 of \cite{SheWer}, a.s.\ there are only finitely many clusters of $\mathcal{L}$ with diameter larger than any positive threshold; moreover they are all at positive distance from each other. Let $\rho\in(0,\delta/2)$ be such that, with high probability, for every $z\in D$ we have that $z\in \Hull C$ for some outermost cluster $C$ of $\BLS$ with $\diam C \geq \delta/2$, or $d_{\mathcal{E}}(z,C) < \delta/4$ for some outermost cluster $C$ of $\BLS$ with $\rho < \diam C < \delta/2$. The existence of such a $\rho$ follows from the fact that a.s.\ $\BLS$ is dense in $D$ and that there are only finitely many clusters of $\BLS$ with diameter at least $\delta/2$. We call a cluster or subcluster \emph{large} (\emph{small}) if its diameter is larger than (less than or equal to) $\rho$.

Let $\varepsilon_1>0$ be such that, with high probability,
\[
| \diam C - \rho | > \varepsilon_1
\]
for all clusters $C$ of $\mathcal{L}$. The existence of such an $\varepsilon_1$ follows from the fact that a.s.\ there are only finitely many clusters with diameter larger than any positive threshold (see Lemma 9.7 of \cite{SheWer}) and the fact that the distribution of cluster sizes does not have atoms. The latter fact is a consequence of scale invariance \cite{LawWer}, which can be seen to yield that the existence of an atom would imply the existence of uncountably many atoms, which is impossible. Let $\varepsilon_2>0$ be such that, with high probability,
\[
d_{\mathcal{E}}(C_1,C_2)>\varepsilon_2
\]
for all distinct large clusters $C_1,C_2$ of $\mathcal{L}$. For every large cluster $C_1$ of $\mathcal{L}$, let $\wp(C_1)$ be a path connecting $\text{Hull} C_1$ with $\infty$ such that, for all large clusters $C_2$ of $\BLS$ such that $\Hull C_1 \not\subset \Hull C_2$, we have that $\wp(C_1) \cap \text{Hull} C_2 =\emptyset$. Let $\varepsilon_3>0$ be such that, with high probability,
\[
d_{\mathcal{E}}(\wp(C_1), \text{Hull}C_2) > \varepsilon_3
\]
for all large clusters $C_1,C_2$ of $\mathcal{L}$ such that $\Hull C_1\not\subset \Hull C_2$. By Lemma \ref{lem:topological} (and Remark \ref{remark:topological}) we can choose $\varepsilon_4>0$ such that, with high probability, for every large cluster $C$ of $\BLS$,
\[
\text{if } d_H(C,\tilde C) < \varepsilon_4, \text{ then } \Ext C \subset (\Ext \tilde C)^{\min\{\delta,\varepsilon_2\}/8}
\]
for any collection of loops $\tilde C$.

Let $t_0>0$ be such that, with high probability, every subcluster $C$ of $\mathcal{L}$ with $\diam C > \rho-\varepsilon_1$ contains a loop of time length larger than $t_0$. Such a $t_0$ exists by Lemma \ref{existencerho}. In particular, every large subcluster of $\mathcal{L}$ contains a loop of time length larger than $t_0$. Note that the number of loops with time length larger than $t_0$ is a.s.\ finite.

From now on the proof is in six steps, and we start by giving a sketch of these steps (see Figure \ref{fig:diagramproof}). First, we treat the large clusters. For every large cluster $C$ of $\mathcal{L}$, we choose a finite subcluster $C'$ of $C$ such that $d_H(C,C')$ and $d_H(\Ext C,\Ext C')$ are small, using Theorem \ref{finiteapproximation}. Second, we approximate $C'$ by a subcluster $\tilde C_N'$ of $\mathcal{\tilde L}_N^{N^{\theta-2}}$ such that $d_H(\Ext C', \Ext \tilde C_N')$ is small, using the one-to-one correspondence between random walk loops and Brownian loops, Theorem \ref{thm:collectionconv} and Corollary \ref{corbrownianloop}. Third, we let $\tilde C_N$ be the cluster of $\mathcal{\tilde L}_N^{N^{\theta-2}}$ that contains $\tilde C_N'$. Here we make sure, using Proposition \ref{correspondence1}, that for distinct subclusters $\tilde C_{1,N}', \tilde C_{2,N}'$, the corresponding clusters $\tilde C_{1,N}, \tilde C_{2,N}$ are distinct. It follows that $d_H(C, \tilde C_N)$ and $d_H(\Ext C,\Ext \tilde C_N)$ are small. Fourth, we show that the obtained clusters $\tilde C_N$ are large. We also show that we obtain in fact all large clusters of $\mathcal{\tilde L}_N^{N^{\theta-2}}$ in this way. Fifth, we prove that a large cluster $C$ of $\BLS$ is outermost if and only if the corresponding large cluster $\tilde C_N$ of $\mathcal{\tilde L}_N^{N^{\theta-2}}$ is outermost. Sixth, we deal with the small outermost clusters.

\begin{figure}
	\begin{center}
		\includegraphics[scale=1]{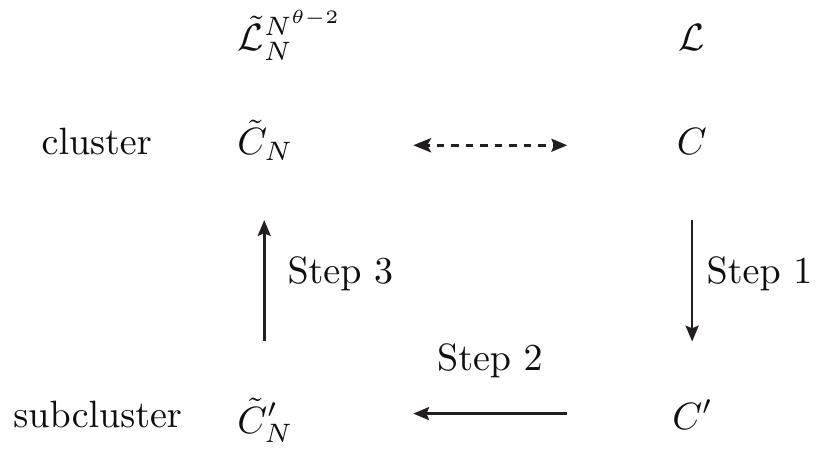}
	\end{center}
	\caption{Schematic diagram of the proof of Theorem \ref{mainthmBLS}. We start with a cluster $C$ of $\BLS$ and, following the arrows, we construct a cluster $\tilde C_N$ of $\mathcal{\tilde L}_N^{N^{\theta-2}}$. The dashed arrow indicates that $C$ and $\tilde C_N$ satisfy \eqref{toprovemain2}.}
	\label{fig:diagramproof}
\end{figure}

{\bf Step 1.} Let $\mathcal{C}$ be the collection of large clusters of $\mathcal{L}$. By Lemma 9.7 of \cite{SheWer}, the collection $\mathcal{C}$ is finite a.s. For every $C\in\mathcal{C}$ let $C'$ be a finite subcluster of $C$ such that $C'$ contains all loops in $C$ which have time length larger than $t_0$ and
\begin{align}
d_H(C,C') &< \min\{\delta,\varepsilon_1,\varepsilon_2,\varepsilon_3,\varepsilon_4\}/16,\label{hdist1} \\
d_H(\Ext C, \Ext C') &< \min\{\delta,\varepsilon_2\}/16,\label{ext1}
\end{align}
a.s. This is possible by Theorem \ref{finiteapproximation}. Let $\mathcal{C}'$ be the collection of these finite subclusters $C'$.

{\bf Step 2.} For every $C'\in\mathcal{C}'$ let $\tilde C_N'\subset \mathcal{\tilde L}_N^{N^{\theta-2}}$ be the set of random walk loops which correspond to the Brownian loops in $C'$, in the one-to-one correspondence from the first paragraph of this proof. This is possible for large $N$, with high probability, since then
\[
\textstyle \bigcup \mathcal{C}' \subset \mathcal{L}^{N^{\theta-2}},
\]
where $\bigcup \mathcal{C}' = \bigcup_{C'\in\mathcal{C}'} C'$. Let $\mathcal{\tilde C}_N'$ be the collection of these sets of random walk loops $\tilde C_N'$.

Now we prove some properties of the elements of $\mathcal{\tilde C}_N'$. By Corollary \ref{corbrownianloop}, $C'$ has no touchings a.s. Hence, by Theorem \ref{thm:collectionconv} (and Remark \ref{remark:topological}), for large $N$, with high probability,
\begin{equation}
d_H(\Ext C',\Ext \tilde C_N') < \min\{\delta,\varepsilon_2\}/16.\label{ext2}
\end{equation}
Next note that almost surely, $d_{\mathcal{E}}(\gamma,\gamma')>0$ for all non-intersecting loops $\gamma,\gamma'\in\mathcal{L}$, and $\text{overlap}(\gamma,\gamma')>0$ for all intersecting loops $\gamma,\gamma'\in\mathcal{L}$. Since the number of loops in $\bigcup \mathcal{C}'$ is finite, we can choose $\eta>0$ such that, with high probability, $d_{\mathcal{E}}(\gamma,\gamma') > \eta$ for all non-intersecting loops $\gamma,\gamma'\in\bigcup \mathcal{C}'$, and $\text{overlap}(\gamma,\gamma') > \eta$ for all intersecting loops $\gamma,\gamma'\in\bigcup \mathcal{C}'$. For large $N$, $cN^{-1}\log N < \eta/2$ and hence with high probability,
\begin{equation}\label{equation4}
\gamma_1\cap\gamma_2 = \emptyset \text{ if and only if } \tilde \gamma_1\cap\tilde\gamma_2 = \emptyset, \text{ for all } \textstyle \gamma_1,\gamma_2\in\bigcup\mathcal{C}',
\end{equation}
where $\tilde \gamma_1,\tilde \gamma_2$ are the random walk loops which correspond to the Brownian loops $\gamma_1,\gamma_2$, respectively. By \eqref{equation4}, every $\tilde C_N'\in\mathcal{\tilde C}_N'$ is connected and hence a subcluster of $\mathcal{\tilde L}_N^{N^{\theta-2}}$. Also by \eqref{equation4}, for distinct $C_1',C_2'\in\mathcal{C}'$, the corresponding $\tilde C_{1,N}', \tilde C_{2,N}'\in\mathcal{\tilde C}_N'$ do not intersect each other when viewed as subsets of the plane.

{\bf Step 3.} For every $\tilde C_N'\in\mathcal{\tilde C}_N'$ let $\tilde C_N$ be the cluster of $\mathcal{\tilde L}_N^{N^{\theta-2}}$ which contains $\tilde C_N'$. Let $\mathcal{\tilde C}_N$ be the collection of these clusters $\tilde C_N$. We claim that for distinct $\tilde C_{1,N}', \tilde C_{2,N}'\in\mathcal{\tilde C}_N'$, the corresponding $\tilde C_{1,N}, \tilde C_{2,N} \in\mathcal{\tilde C}_N$ are distinct, for large $N$, with high probability. This implies that there is one-to-one correspondence between elements of $\mathcal{\tilde C}_N'$ and elements of $\mathcal{\tilde C}_N$, and hence between elements of $\mathcal{C}$, $\mathcal{C}'$, $\mathcal{\tilde C}_N'$ and $\mathcal{\tilde C}_N$.

To prove the claim, we combine Proposition \ref{correspondence1} and the one-to-one correspondence between random walk loops and Brownian loops to obtain that, for large $N$, with high probability,
\begin{equation}\label{equation3}
\text{if } \gamma_1\cap\gamma_2 = \emptyset \text{ then } \tilde \gamma_1\cap\tilde\gamma_2 = \emptyset, \text{ for all } \gamma_1,\gamma_2\in\mathcal{L}^{N^{\theta-2}},
\end{equation}
where $\tilde\gamma_1,\tilde\gamma_2$ are the random walk loops which correspond to the Brownian loops $\gamma_1,\gamma_2$, respectively. Let $\tilde C_{1,N}', \tilde C_{2,N}'\in\mathcal{\tilde C}_N'$ be distinct. Let $C_1',C_2'\in\mathcal{C}'$ be the finite subclusters of Brownian loops which correspond to $\tilde C_{1,N}', \tilde C_{2,N}'$, respectively. By construction, $C_1',C_2'$ are contained in clusters of $\mathcal{L}^{N^{\theta-2}}$ which are distinct. Hence by \eqref{equation3}, $\tilde C_{1,N}, \tilde C_{2,N}$ are distinct.

Next we prove that, for large $N$, with high probability,
\begin{align}
d_H(C, \tilde C_N) &< \min\{\delta,\varepsilon_1,\varepsilon_2,\varepsilon_3,\varepsilon_4\}/4,\label{hdist3} \\
d_H(\text{Ext} C, \text{Ext} \tilde C_N) &< \min\{\delta,\varepsilon_2\}/4,\label{ext3dist}
\end{align}
which implies that $C$ and $\tilde C_N$ satisfy \eqref{toprovemain2}. To prove \eqref{hdist3}, let $N$ be sufficiently large, so that in particular $cN^{-1}\log N < \min\{ \delta,\varepsilon_1,\varepsilon_2,\varepsilon_3,\varepsilon_4 \}/16$. By \eqref{hdist1}, with high probability,
\[
C \subset (C')^{\min\{ \delta,\varepsilon_1,\varepsilon_2,\varepsilon_3,\varepsilon_4 \}/16}
\subset (\tilde C_N')^{\min\{ \delta,\varepsilon_1,\varepsilon_2,\varepsilon_3,\varepsilon_4 \}/8}
\subset (\tilde C_N)^{\min\{ \delta,\varepsilon_1,\varepsilon_2,\varepsilon_3,\varepsilon_4 \}/8}.
\]
By \eqref{equation3}, $\tilde C_N \subset C^{\min\{ \delta,\varepsilon_1,\varepsilon_2,\varepsilon_3,\varepsilon_4 \}/16}$. This proves \eqref{hdist3}. To prove \eqref{ext3dist}, note that by \eqref{hdist3} and the definition of~$\varepsilon_4$, $\text{Ext} C \subset (\text{Ext} \tilde C_N)^{\min\{\delta,\varepsilon_2\}/8}$. By \eqref{ext1} and \eqref{ext2},
\[
\text{Ext} \tilde C_N \subset \text{Ext} \tilde C_N' \subset (\text{Ext} C)^{\min\{\delta,\varepsilon_2\}/8}.
\]
This proves \eqref{ext3dist}.

{\bf Step 4.} We prove that, for large $N$, with high probability, all $\tilde C_N \in \mathcal{\tilde C}_N$ are large, and that all large clusters of $\mathcal{\tilde L}_N^{N^{\theta-2}}$ are elements of $\mathcal{\tilde C}_N$. This gives that, for large $N$, with high probability, there is a one-to-one correspondence between large clusters $C$ of $\BLS$ and large clusters $\tilde C_N$ of $\mathcal{\tilde L}_N^{N^{\theta-2}}$ such that \eqref{hdist3} and \eqref{ext3dist} hold, and hence such that \eqref{toprovemain2} holds.

First we show that, for large $N$, with high probability, all $\tilde C_N \in \mathcal{\tilde C}_N$ are large. By \eqref{hdist3} and the definition of $\varepsilon_1$, for large $N$, with high probability, $\diam\tilde C_N > \diam C - \varepsilon_1 > \rho$, i.e.\ $\tilde C_N$ is large.

Next we prove that, for large $N$, with high probability, all large clusters of $\mathcal{\tilde L}_N^{N^{\theta-2}}$ are elements of $\mathcal{\tilde C}_N$. Let $\tilde G_N$ be a large cluster of $\mathcal{\tilde L}_N^{N^{\theta-2}}$. Let $G_N\subset \mathcal{L}^{N^{\theta-2}}$ be the set of Brownian loops which correspond to the random walk loops in $\tilde G_N$. By \eqref{equation3}, $G_N$ is connected and hence a subcluster of $\mathcal{L}$. If $cN^{-1}\log N < \varepsilon_1/2$, then $\diam G_N > \rho-\varepsilon_1$. Let $G$ be the cluster of $\mathcal{L}$ which contains $G_N$. We have that $\diam G > \rho-\varepsilon_1$ and hence by the definition of $\varepsilon_1$, with high probability, $G$ is large, i.e.\ $G\in\mathcal{C}$. Let $\tilde G_N^*$ be the element of $\mathcal{\tilde C}_N$ which corresponds to $G$. We claim that
\begin{equation}\label{starisG}
\tilde G_N^* = \tilde G_N,
\end{equation}
which implies that $\tilde G_N \in \mathcal{\tilde C}_N$.

To prove \eqref{starisG}, let $G'$ be the element of $\mathcal{C}'$ which corresponds to $G$. Since $G_N$ is a subcluster of $\mathcal{L}$ with $\diam G_N > \rho - \varepsilon_1$, $G_N$ contains a loop $\gamma$ of time length larger than $t_0$. Since $\gamma\in G$ and $t_{\gamma} > t_0$, by the construction of $G'$, we have that $\gamma\in G'$. Hence $\tilde \gamma \in \tilde G_N^*$, where $\tilde \gamma$ is the random walk loop corresponding to the Brownian loop $\gamma$. Since $\gamma\in G_N$, by the definition of $G_N$, we have that $\tilde \gamma\in \tilde G_N$. It follows that $\tilde\gamma \in \tilde G_N \cap \tilde G_N^*$, which implies that \eqref{starisG} holds.

{\bf Step 5.} Let $C,G$ be distinct large clusters of $\BLS$, and let $\tilde C_N,\tilde G_N$ be the large clusters of $\mathcal{\tilde L}_N^{N^{\theta-2}}$ which correspond to $C,G$, respectively. We prove that, for large $N$, with high probability,
\begin{equation}\label{eqoutermost}
\Hull C\subset \Hull G \text{ if and only if } \Hull \tilde C_N \subset \Hull \tilde G_N.
\end{equation}
It follows from \eqref{eqoutermost} that a large cluster $C$ of $\BLS$ is outermost if and only if the corresponding large cluster $\tilde C_N$ of $\mathcal{\tilde L}_N^{N^{\theta-2}}$ is outermost.

To prove \eqref{eqoutermost}, suppose that $\Hull C\subset \Hull G$. By the definition of $\varepsilon_2$, $(\text{Hull} C)^{\varepsilon_2/2} \subset \IC \setminus (\text{Ext} G)^{\varepsilon_2/2}$. By \eqref{ext3dist}, $\text{Ext} \tilde G_N \subset (\text{Ext} G)^{\varepsilon_2/2}$. By \eqref{hdist3}, $\tilde C_N \subset C^{\varepsilon_2/2} \subset (\text{Hull} C)^{\varepsilon_2/2}$. Hence
\[
\tilde C_N \subset (\Hull C)^{\varepsilon_2/2} \subset \IC \setminus (\Ext G)^{\varepsilon_2/2} \subset \IC \setminus \Ext \tilde G_N = \Hull \tilde G_N.
\]
It follows that $\Hull \tilde C_N \subset \Hull \tilde G_N$.

To prove the reverse implication of \eqref{eqoutermost}, suppose that $\Hull \tilde C_N \subset \Hull \tilde G_N$. There are three cases: $\Hull C\subset \Hull G$, $\Hull G\subset \Hull C$ and $\Hull C\cap\Hull G=\emptyset$. We will show that the second and third case lead to a contradiction, which implies that $\Hull C\subset \Hull G$. For the second case, suppose that $\Hull G\subset \Hull C$. Then, by the previous paragraph, $\Hull \tilde G_N\subset \Hull \tilde C_N$. This contradicts the fact that $\Hull \tilde C_N \subset \Hull \tilde G_N$ and $\tilde C_N \cap \tilde G_N=\emptyset$.

\begin{figure}
	\begin{center}
		\includegraphics[scale=0.9]{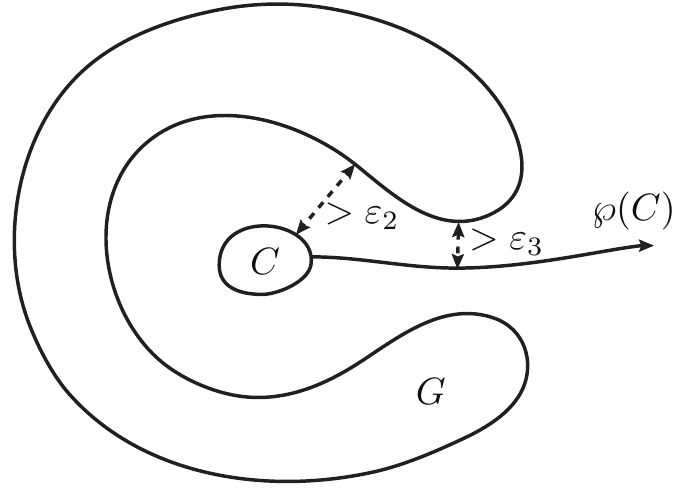}
	\end{center}
	\caption{The case $\Hull C \cap \Hull G =\emptyset$ in Step 5.}
	\label{fig:clusterpath}
\end{figure}

For the third case, suppose that $\Hull C \cap \Hull G=\emptyset$. Let $\wp(C)$ be the path from the definition of $\varepsilon_3$, which connects $\Hull C$ with $\infty$ such that $\wp(C)\cap \Hull G=\emptyset$ (see Figure \ref{fig:clusterpath}). By the definition of $\varepsilon_2,\varepsilon_3$, with high probability,
\[
((\Hull C)^{\min\{\varepsilon_2,\varepsilon_3\}/2} \cup \wp(C)) \cap (\Hull G)^{\min\{\varepsilon_2,\varepsilon_3\}/2} = \emptyset.
\]
By \eqref{hdist3}, for large $N$, with high probability,
\[
\tilde C_N\subset C^{\min\{\varepsilon_2,\varepsilon_3\}/2} \subset (\Hull C)^{\min\{\varepsilon_2,\varepsilon_3\}/2}.
\]
Similarly, $\tilde G_N\subset (\Hull G)^{\min\{\varepsilon_2,\varepsilon_3\}/2}$. It follows that there exists a path from $\tilde C_N$ to $\infty$ that avoids $\tilde G_N$. This contradicts the assumption that $\Hull \tilde C_N\subset \Hull \tilde G_N$.

{\bf Step 6.} Finally we treat the small outermost clusters. Let $G$ be a small outermost cluster of $\BLS$. By the definition of $\rho$, with high probability, there exists an outermost cluster $C$ of $\BLS$ with $\rho<\diam C < \delta/2$ such that $d_{\mathcal{E}}(C,G) < \delta/4$. It follows that
\begin{align*}
& d_H(C,G) \leq d_{\mathcal{E}}(C,G) + \max \{\diam C,\diam G\} < 3\delta/4,\\
& d_H(\Ext C,\Ext G) \leq \textstyle \frac{1}{2} \max \{ \diam C,\diam G\} < \delta/4.
\end{align*}
Note that $C$ is large, and let $\tilde C_N$ be the large outermost cluster of $\mathcal{\tilde L}_N^{N^{\theta-2}}$ which corresponds to $C$. Since $C$ and $\tilde C_N$ satisfy \eqref{hdist3} and \eqref{ext3dist}, we obtain that $d_H(G,\tilde C_N)<\delta$ and $d_H(\Ext G,\Ext \tilde C_N)<\delta/2$.

Next, by the one-to-one correspondence between elements of $\mathcal{C}$ and $\mathcal{\tilde C}_N$ satisfying \eqref{hdist3} and \eqref{ext3dist}, for large $N$, with high probability,
\begin{align} \label{eq:exteriordistance}
\textstyle{\dist_H \big( \bigcap_{C\in\mathcal{C}} \Ext C, \bigcap_{\tilde C_N \in \mathcal{\tilde C}_N} \Ext \tilde C_N \big) <\delta/4.}
\end{align}
Let $\tilde G_N$ be a small outermost cluster of $\mathcal{\tilde L}_N^{N^{\theta-2}}$, then we have that $\tilde G_N \subset \bigcap_{\tilde C_N \in \mathcal{\tilde C}_N} \Ext \tilde C_N$. By~\eqref{eq:exteriordistance} and the fact that $\BLS$ is dense in $D$, a.s.\ there exists an outermost cluster $C$ of $\BLS$ with $\diam C < \delta/2$ such that $d_{\mathcal{E}}(C, \tilde G_N) < \delta/2$. It follows that
\begin{align*}
& d_H(C,\tilde G_N) \leq d_{\mathcal{E}}(C,\tilde G_N) + \max \{\diam C,\diam \tilde G_N\} < \delta,\\
& d_H(\Ext C,\Ext \tilde G_N) \leq \textstyle \frac{1}{2} \max \{ \diam C,\diam \tilde G_N\} < \delta/4.
\end{align*}
This completes the proof.
\end{proof}

\noindent
{\bf Acknowledgments.} Tim van de Brug and Marcin Lis thank New York University Abu Dhabi for the hospitality during two visits in 2013 and 2014. Tim van de Brug and Federico Camia thank Gregory Lawler for useful conversations in Prague in 2013 and in Seoul in 2014, respectively. The authors thank Ren\'e Conijn for a useful remark concerning the proof of Theorem \ref{thmbrownianmotion}. The research was conducted and the paper was completed while Marcin Lis was at the Department of Mathematics of VU University Amsterdam. At the time of the research all authors were supported by NWO Vidi grant 639.032.916.

\bibliographystyle{amsplain}
\bibliography{rwls}

\enlargethispage{0.2cm}

\end{document}